\documentclass[a4paper,twoside]{amsart}
\usepackage[utf8]{inputenc}
\usepackage[hidelinks]{hyperref}
\usepackage{graphicx}
\usepackage{enumitem}
\usepackage{amssymb}

\theoremstyle{plain}
\newtheorem{theorem}{Theorem}[section]
\newtheorem{corollary}[theorem]{Corollary}
\newtheorem{prop}[theorem]{Proposition}
\newtheorem{observation}[theorem]{Observation}
\newtheorem{lemma}[theorem]{Lemma} 

\newtheorem{question}[theorem]{Question}

\theoremstyle{definition}
\newtheorem{definition}[theorem]{Definition}
\newtheorem{example}[theorem]{Example}
\newtheorem*{remark*}{Remark}
\newtheorem*{claim*}{Claim}
\theoremstyle{remark}
\newtheorem{remark}[theorem]{Remark}

\def\dom{\mathop{\mathrm{Dom}}\nolimits}

\def\str#1{\mathbf {#1}}
\def\Fraisse{Fra\"{\i}ss\' e}
\def\range{\mathop{\mathrm{Range}}\nolimits}

\def\Aut{\mathop{\mathrm{Aut}}\nolimits}
\def\age{\mathop{\mathrm{Age}}\nolimits}

\renewcommand{\restriction}{\mathord{\upharpoonright}}
\def\marginpar#1{}

\begin{document}
\bibliographystyle{alpha}

\title[]{Extension property for partial automorphisms of the $n$-partite and semigeneric tournaments}

\authors{
\author[J. Hubi\v cka]{Jan Hubi\v cka}
\address{Department of Applied Mathematics (KAM)\\ Charles University\\ Prague, Czech Republic}
\email{hubicka@kam.mff.cuni.cz}
\author[C. Jahel]{Colin Jahel}
\address{Institute of Algebra\\TU Dresden\\ Dresden, Germany}
\email{colin.jahel@tu-dresden.de}
\author[M. Kone\v cn\'y]{Mat\v ej Kone\v cn\'y}
\address{Institute of Algebra\\TU Dresden\\ Dresden, Germany \\and 
Department of Applied Mathematics (KAM)\\ Charles University\\ Prague, Czech Republic}
\email{matej.konecny@tu-dresden.de}
\author[M. Sabok]{Marcin Sabok}
\address{Department of Mathematics and Statistics\\ McGill University\\ Montreal, Canada}
\email{marcin.sabok@mcgill.ca}

\thanks{J. H. and M. K. were supported by the project 21-10775S of the Czech Science Foundation (GA\v CR). Part of the work has been done when J. H. and M. K. visited M. S. thanks to the support of the project RandNET which has received funding from the European Union’s Horizon 2020 research and innovation programme under the Marie Skłodowska-Curie grant agreement No 101007705. C. J. was supported by ANR project AGRUME (ANR-17-CE40-0026), and by the Deutsche Forschungsgemeinschaft (DFG, German Research Foundation) -- project number 467967530. M. S. was funded by the NSERC Discovery Grant RGPIN-2020-05445, NSERC Discovery Accelerator Supplement RGPAS-2020-00097 and the NCN Grant Harmonia 2018/30/M/ST1/00668.}
}
\date{}

\begin{abstract}
We present a proof of the extension property for partial
automorphisms (EPPA) for classes of finite $n$-partite tournaments for $n \in \{2,3,\ldots,\omega\}$,
and for the class of finite semigeneric tournaments. We also prove that the generic $\omega$-partite tournament and the generic semigeneric tournament have ample generics.
\end{abstract}
\maketitle
\section{Introduction}
A class of structures $\mathcal C$ has the \emph{extension property for
partial automorphism} (\emph{EPPA}), sometimes also called the \emph{Hrushovski
property}, if for every $\str A\in \mathcal C$ there exists $\str B\in \mathcal C$
containing $\str A$ as a substructure with the property that every isomorphism
of two substructures of $\str A$ (also called a \emph{partial automorphism} of $\str A$)
extends to an automorphism of $\str B$. We call $\str B$ with such a property an {\em EPPA-witness of $\str A$}.

A \emph{directed graph} is a structure with one binary relation which is irreflexive (i.e. no loops) and antisymmetric (i.e. no bi-directional edges).\footnote{Note that this is sometimes called an \emph{oriented graph}, while a directed graph is sometimes allowed to have loops and/or bi-directional edges. We decided to follow the nomenclature which is standard in the context of homogeneous structures.} If, in a directed graph, the pair $(x,y)$ is in the relation, we say that there \emph{is an edge between $x$ and $y$}, or that \emph{$x$ and $y$ are adjacent}, and that this edge \emph{goes} (or \emph{is oriented}) from $x$ to $y$. Given $n \in \{2,3,\ldots,\omega\}$, a directed graph $\str A$ is an \emph{$n$-partite tournament} if its vertex set can
be partitioned into (possibly empty) pairwise disjoint sets $A_1\cup A_2 \cup\ldots \cup A_n=A$ (called \emph{parts}) such that
every pair of vertices in different parts is connected by a directed edge and
there are no edges between vertices in same part. An $\omega$-partite tournament $\str A$ is \emph{semigeneric} if for every pair of parts $X$ and $Y$ of $\str A$, and every $a\neq b\in X$ and $c\neq d\in Y$
it holds that the number of edges directed from $\{a,b\}$ to $\{c,d\}$ is even.

In 1992, Hrushovski~\cite{hrushovski1992} established that the class of all finite graphs has EPPA. This result was used by Hodges, Hodkinson, Lascar, and Shelah to show the small index property for the random graph~\cite{hodges1993b}. After this, the quest of studying EPPA continued with a series of papers 
including~\cite{Herwig1995,herwig1998,herwig2000,hodkinson2003,solecki2005,vershik2008,Aranda2017,Hubicka2017sauer,Conant2015,Hubicka2018metricEPPA,Konecny2018b,Konecny2019a,eppatwographs,otto2017,Evans3,Hubicka2018EPPA,BradleyEPPAnumbers}.

In this paper we contribute to this quest by giving proofs of the following two theorems:
\begin{theorem}
\label{thm:partite}
For every $n\in \{2,3,\ldots,\omega\}$ the class of all finite $n$-partite tournaments has EPPA.
\end{theorem}

\begin{theorem}
\label{thm:semigeneric}
The class of all finite semigeneric tournaments has EPPA.
\end{theorem}
These theorems have been announced in a Eurocomb extended abstract~\cite{HubickaSemigenericAMUC}.

By Proposition~6.4 of~\cite{Kechris2007}, EPPA for a class $\mathcal C$ is equivalent to the automorphism group of the \Fraisse{} limit of $\mathcal C$ being the closure of a chain of compact subgroups. It is well-known (see e.g. Proposition~G.2.2 of~\cite{Bekka2008}) that this implies amenability of the group, hence we get the following corollary which was first proved by Pawliuk and Soki\'c by different methods:
\begin{corollary}[\cite{PawliukSokic16}]
Let $G$ be the automorphism group of the \Fraisse{} limit of $\mathcal C$ for $\mathcal C$ either the class of all finite $n$-partite tournaments, $n\in\{2,3,\ldots,\omega\}$, or the class of all finite semigeneric tournaments. Then $G$ is amenable.
\end{corollary}
Here, a topological group $G$ is \emph{amenable} if every $G$-flow admits a left-invariant probability measure.

\begin{definition}[\cite{Truss1992,hodges1993b,Kechris2007}]\label{defn:ample}
Let $\str M$ be a countable structure and let $n\geq 1$ be an integer. We say that $\str M$ has \emph{$n$-generic automorphisms} if $G = \Aut(\str M)$ has a comeagre orbit on $G^n$ in its action by diagonal conjugation. We say that $\str M$ has \emph{ample generics} if it has $n$-generic automorphisms for every $n\geq 1$.
\end{definition}
Here, the action by diagonal conjugation is defined by $$g\cdot (h_1,\ldots,h_n) = (gh_1g^{-1},\ldots,gh_ng^{-1}).$$

In order to prove the small index property for the random graph, Hodges, Hodkinson, Lascar, and Shelah actually proved that the random graph has ample generics and that this fact implies the small index property. Kechris and Rosendal~\cite{Kechris2007} subsequently extracted an equivalent combinatorial condition for ample generics, streamlined the arguments and proved several more consequences of ample generics. See Section~\ref{sec:ample:background} for more details.

We prove:
\begin{theorem}
\label{thm:npart_ample}
The \Fraisse{} limit of the class of all finite $\omega$-partite tournaments has ample generics.

For every $n\in \{2,3,\ldots\}$, the \Fraisse{} limit of the class of all finite $n$-partite tournaments does not have 1-generic automorphisms. In fact, its automorphism group does not even have a dense conjugacy class.
\end{theorem}

\begin{theorem}
\label{thm:semig_ample}
The \Fraisse{} limit of the class of all finite semigeneric tournaments has ample generics.
\end{theorem}

By~\cite{Kechris2007}, Theorems~\ref{thm:npart_ample} and~\ref{thm:semig_ample} immediately give us the following corollary (for definitions see Chapter~1.6 of~\cite{Kechris2007}):
\begin{corollary}
Let $G$ be the automorphism group of the \Fraisse{} limit of $\mathcal C$ for $\mathcal C$ either the class of all finite $\omega$-partite tournaments, or the class of all finite semigeneric tournaments. Then:
\begin{enumerate}
  \item $G$ has the small index property.
  \item $G$ has uncountable cofinality.
  \item $G$ has properties (FA) and (FH).
  \item $G$ has the 21-Bergman property.
\end{enumerate}
\end{corollary}

\subsection{Homogeneous structures}
One can show that if a hereditary isomorphism-closed class $\mathcal C$ of finite relational structures has countably many members up to isomorphism, the \emph{joint embedding property} (that is, for every $\str A, \str B \in \mathcal C$ there is $\str C\in \mathcal C$ which embeds both $\str A$ and $\str B$), and EPPA then it has the \emph{amalgamation property} (that is, for every $\str A, \str B_1, \str B_2\in \mathcal C$ and embeddings $\alpha_i\colon \str A \to \str B_i$ for $i\in \{1,2\}$, there is $\str C \in \mathcal C$ with embeddings $\beta_i\colon \str B_i \to \str C$ for $i\in \{1,2\}$ such that $\beta_1\alpha_1 = \beta_2\alpha_2$), see for example~\cite[Proposition 7.1]{hubicka2025twenty}.

Given a structure $\str M$, its \emph{age}, denoted by $\age(\str M)$, is the class of all finite structures which embed to $\str M$. By the \Fraisse{} theorem~\cite{Fraisse1953,Fraisse1986}, the age of every homogeneous structure has the joint embedding property and the amalgamation property, and conversely, a hereditary isomorphism-closed class with the joint embedding property, the amalgamation property, and countably many members up to isomorphism is the age of a countable homogeneous structure, its \emph{\Fraisse{} limit}, which is unique up to isomorphism. Here, a structure $\str M$ is \emph{homogeneous} if every isomorphism between finite substructures of $\str M$ extends to an automorphism of $\str M$ (so, in a way, it is an EPPA-witness for itself). This restricts the candidate classes for EPPA to those provided by the classification programme of homogeneous structures (see e.g.~\cite{Lachlan1980,Lachlan1984,lachlan1984countable,Cherlin1998,Cherlin1999,Cherlin2013}). 

\medskip

Semigeneric tournaments and $n$-partite tournaments appear in Cherlin's classification
of countable homogeneous directed graphs~\cite{Cherlin1998} and are one of the few known examples of amalgamation classes having EPPA where this property does not follow by a direct application of the Herwig--Lascar theorem~\cite{herwig2000}, or more generally, its strengthening by Hubička, Konečný and Nešetřil~\cite{Hubicka2018EPPA} (in addition to two-graphs~\cite{eppatwographs}, finite permutation groups~\cite{Siniora2}, and certain antipodal metrically homogeneous graphs~\cite{Konecny2019a}).  As discussed in~\cite{PawliukSokic16}, Theorems~\ref{thm:partite} and~\ref{thm:semigeneric} imply that in order to fully classify which homogeneous directed graphs have EPPA, it now remains to decide EPPA for the class of all finite tournaments, the class of all finite
directed graphs omitting an independent set of size $k$ ($k\geq 2$), and the class of all finite double covers of tournaments.

In particular, EPPA for tournaments is a long standing open problem with important connections to group theory which was posed in 2000 by Herwig and Lascar~\cite{herwig2000}, see also~\cite{Sabok} for some recent progress on this question. We identify a weakening of the questions whether tournaments or directed graphs without large independent sets have EPPA:

\begin{question}\label{q:tournaments}
For which $k\geq 2$ is there $\ell$ such that for every directed graph $\str A$ which contains no independent set of size $k$ there is a directed graph $\str B$ which contains no independent set of size $\ell$ such that $\str A\subseteq \str B$ and every partial automorphism of $\str A$ extends to an automorphism of $\str B$.
\end{question}
Note that for $\ell = k$ this is simply EPPA for directed graphs with no independent set of size $k$. In particular, for $k=2$ these are tournaments. While we seem to encounter the same obstacles when trying to adapt the existing methods to answer Question~\ref{q:tournaments} as when trying to prove EPPA for tournaments, we believe that this weakening might be more approachable and more robust, as it seems to non-trivially relax the group-theoretical constraints. For example, tournaments only have automorphisms of odd order, but this is no longer the case for directed graph without large independent sets.

\section{Warm-up: EPPA for graphs}\label{sec:graphs}
Our constructions of EPPA-witnesses for $n$-partite tournaments and semigeneric tournaments will have very similar structure: Given a finite structure $\str A$, we will first construct a finite structure $\str B$ and give an embedding $\psi\colon\str A\to \str B$. Our goal will be to eventually prove that $\str B$ is an EPPA-witness for $\psi(\str A)$. Towards this, we first define some families of nice automorphisms of $\str B$ which we will later use in extending partial automorphisms of $\str A$. Next, given a partial automorphism of $\str A$, we will show which nice automorphisms of $\str B$ we need to compose in order to extend it, and we prove that our constructions indeed work.

To help the reader familiarize themselves with this structure of the proof, we start with a proof of EPPA for graphs. This was first proved by Hrushovski~\cite{hrushovski1992}, here we show (a slightly different presentation of) the proof by Hubička, Konečný, and Nešetřil from~\cite{Hubicka2018EPPA}.

Fix a graph $\str{A}$ with vertex set $A=\{1,\ldots, n\}$. 

\subsection{Witness construction}
Given vertex $x\in A$, we will call every function $\chi\colon A\setminus\{x\}\to \mathbb Z_2$ a \emph{valuation function for $x$}. Note that operations with valuation functions will be performed in $\mathbb Z_2$. The structure $\str B$ will be constructed as follows:

\begin{enumerate}
 \item The vertex set $B$ consists of all pairs $(x,\chi)$ where $x\in A$ and $\chi\colon A\setminus\{x\}\to \mathbb Z_2$ is a valuation function for $x$, and
 \item $(x,\chi)$ and $(x',\chi')$ are adjacent if and only if $x\neq x'$ and $\chi(x') + \chi'(x) = 0$ (recall that the addition is performed in $\mathbb Z_2$).
\end{enumerate}

Next we construct an embedding $\psi\colon \str A\to \str B$, putting $\psi(x) = (x,\chi_x)$ where $\chi_x\colon A\setminus\{x\} \to \mathbb Z_2$ satisfies
$\chi_x(y)=1$ if $x > y$ and $\{x,y\}\notin E_\str{A}$ and $\chi_x(y)=0$ otherwise (remember that we enumerated $A=\{1,\ldots, n\}$).
It is easy to verify that $\psi$ is indeed an embedding of $\str A$ into $\str B$.
Put $\str A' = \psi(\str A)$. This is the copy of $\str A$ in $\str B$ whose partial automorphisms we will extend.

\subsection{Automorphisms of $\str B$}
We now define two families of automorphisms of $\str B$, denoted by $\theta_\pi$ and $\theta_{u,v}$ respectively, which we will later use to extend partial automorphisms.

Let $\pi\colon A\to A$ be a bijection. Define $\theta_\pi\colon B\to B$ such that $\theta_\pi((x,\chi)) = (\pi(x),\chi^\pi)$ where $\chi^\pi\colon A\setminus\{\pi(x)\}\to \mathbb Z_2$ is the vertex valuation function for $\pi(x)$ satisfying $\chi^\pi(\pi(y)) = \chi(y)$.
\begin{lemma}
For every bijection $\pi\colon A\to A$, $\theta_\pi$ is an automorphism of $\str B$.
\end{lemma}
\begin{proof}
Clearly, $\theta_\pi$ is a bijection, as $\theta_\pi^{-1} = \theta_{\pi^{-1}}$. Note that for every $(x,\chi),(y,\xi) \in B$ we have
$$\chi^\pi(\pi(y)) + \xi^\pi(\xi(x)) = \chi(y) + \xi(x).$$
Consequently, it is indeed true that $(x,\chi),(y,\xi)$ is an edge of $\str B$ if and only if $\theta_\pi((x,\chi)),\theta_\pi((y,\xi))$ is.
\end{proof}

Given $u < v\in A$, define $\theta_{u,v}\colon B\to B$ such that $\theta_{u,v}((x,\chi)) = (x,\chi^{u,v})$ where $\chi^{u,v}\colon A\setminus\{x\}\to \mathbb Z_2$ is the vertex valuation function for $x$ satisfying 
$$\chi^{u,v}(y) = \begin{cases}
1+\chi(y) & \text{ if } \{x,y\} = \{u,v\}\\
\chi(y) & \text{ otherwise}.
\end{cases}$$
\begin{lemma}
For every pair of vertices $u < v\in A$ it holds that $\theta_{u,v}$ is an automorphism of $\str B$.
\end{lemma}
\begin{proof}
Since $\theta_{u,v}^{-1} = \theta_{u,v}$, it is a bijection. Let $(x,\chi)$ and $(y,\xi)$ be vertices of $\str B$. If $\{x,y\} = \{u,v\}$ then
$$\chi^{u,v}(y) + \xi^{u,v}(x) = 1+\chi(y) + 1+\xi(x) = \chi(y)+\xi(x),$$
and if $\{x,y\} \neq \{u,v\}$ then
$$\chi^{u,v}(y) + \xi^{u,v}(x) = \chi(y)+\xi(x).$$
Consequently, $(x,\chi),(y,\xi)$ is an edge of $\str B$ if and only if $\theta_{u,v}((x,\chi)),\theta_{u,v}((y,\xi))$ is, and so $\theta_{u,v}$ is an automorphism of $\str B$.
\end{proof}
\begin{observation}\label{obs:graph:com}
For every quadruple $u,v,w,x\in A$ for which $\theta_{u,v}$ and $\theta_{w,x}$ are defined it holds that $\theta_{u,v}\theta_{w,x} = \theta_{w,x}\theta_{u,v}$. \qed
\end{observation}
\begin{observation}\label{obs:graph:com2}
For every bijection $\pi\colon A\to A$ and for every $u<v\in A$ it holds that $\theta_\pi\theta_{u,v} = \theta_{\pi(u),\pi(v)}\theta_\pi$. \qed
\end{observation}

\subsection{Extending partial automorphisms}
We now show that $\str B$ extends all partial automorphisms of $\str A'$. Fix a partial automorphism $\varphi\colon \str A'\to \str A'$. Looking at the first coordinates, it induces a partial bijection of $A$. Extend it to a bijection $\hat\varphi\colon A\to A$

\begin{observation}\label{obs:graph:proj_fine}
For every $(x,\chi_x) \in \dom(\varphi)$ we have that if $\varphi((x,\chi_x)) = (y,\chi_y)$ then $\theta_{\hat\varphi}((x,\chi_x)) = (y,\chi)$ for some $\chi$. \qed
\end{observation}
In the following we will show that one can compose $\theta_{\hat\varphi}$ with a suitable set of automorphisms $\theta_{u,v}$ to extend $\varphi$. Note that if we believe that this is possible then there is an obvious canonical way of doing it: Given $(x,\chi_x) \in \dom(\varphi)$, with $\varphi((x,\chi_x)) = (y,\chi_y)$ and $\theta_{\hat\varphi}((x,\chi_x)) = (y,\chi)$, we need to fix those entries $z$ for which $\chi_y(z) \neq \chi(z)$.

First, we define a set $F$ consisting of all pairs $\{u,v\}$ satisfying the following:
\begin{enumerate}
  \item $u < v \in A$,
  \item either $(u,\chi_u) \in \dom(\varphi)$ or $(v,\chi_v)\in \dom(\varphi)$ (or both),
  \item\label{g:cond3} if $(u,\chi_u) \in \dom(\varphi)$, $\varphi((u,\chi_u)) = (x,\chi_x)$, and $\theta_{\hat\varphi}((u,\chi_u)) = (x,\chi)$ then $\chi(\hat\varphi(v)) \neq \chi_x(\hat\varphi(v))$, and
  \item\label{g:cond4} if $(v,\chi_v) \in \dom(\varphi)$, $\varphi((v,\chi_v)) = (y,\chi_y)$, and $\theta_{\hat\varphi}((v,\chi_v)) = (y,\xi)$ then $\xi(\hat\varphi(u)) \neq \chi_y(\hat\varphi(u))$.
\end{enumerate}
\begin{observation}
If $u < v\in A$ and $(u,\chi_u), (v,\chi_v) \in \dom(\varphi)$ then~\ref{g:cond3} is satisfied if and only if~\ref{g:cond4} is.
\end{observation}
\begin{proof}
This follows from the fact that $\theta_{\hat\varphi}$ is an automorphism of $\str B$ and the fact that $\varphi$ is a partial automorphism (and thus it preserves whether $(u,\chi_u)(v,\chi_v)$ is an edge): By the definition of $\str B$, knowing $\chi(\hat\varphi(v))$ or $\chi_x(\hat\varphi(v))$ respectively determines $\xi(\hat\varphi(u))$ resp. $\chi_y(\hat\varphi(u))$ and vice versa once we know whether $(u,\chi_u)(v,\chi_v)$ is an edge.
\end{proof}

Let $\theta_F$ be the composition of all $\theta_{u,v}$ for $\{u,v\}\in F$ (by Observation~\ref{obs:npart:com}, $\theta_F$ does not depend on the order of the composition). Put $\theta = \theta_{\hat\varphi}\theta_F$.

\begin{prop}
$\theta$ is an automorphism of $\str B$ extending $\varphi$.
\end{prop}
\begin{proof}
The fact that $\theta$ is an automorphism is clear as it is the composition of several automorphisms of $\str B$. Observation~\ref{obs:graph:proj_fine} gives us that $\theta$ and $\varphi$ agree on the first coordinate. From the definition of $F$ it follows that if $(u,\chi_u)\in \dom(\varphi)$, $\varphi((u,\chi_u)) = (x,\chi_x)$, and $\theta_{\hat\varphi}((u,\chi_u)) = (x,\chi)$, then $\{u,v\} \in F$ if and only if $\chi_x(v) \neq \chi(v)$. Consequently, $\theta$ indeed extends $\varphi$.
\end{proof}

We are now ready to prove Hrushovski's theorem:
\begin{theorem}[Hrushovski, 1992~\cite{hrushovski1992}]
The class of all finite graphs has EPPA.
\end{theorem}
\begin{proof}
Given a finite graph $\str A$, use the construction from this section to construct a finite graph $\str B$. We have proved that $\str B$ is an EPPA-witness for $\str A' = \psi(\str A)$. Clearly, by taking an isomorphism, one gets an EPPA-witness for $\str A$.
\end{proof}

\medskip

Note that $\str B$ only depends on the set $A$ (in fact, only its cardinality) and not the graph $\str A$. Our constructions in the following sections will have similar properties. Note also that $\str B$ can be thought of as some form of product (indeed, its vertices are pairs), and automorphisms $\theta_\pi$ come from some group action on the first coordinate, while automorphisms $\theta_{u,v}$ can be thought of as coming from some group action on the second coordinate. Again, our constructions in the following sections will behave very similarly.

\section{EPPA for $n$-partite tournaments}\label{sec:npart}
In this section we will prove Theorem~\ref{thm:partite}. The structure of the proof will follow what we have seen in Section~\ref{sec:graphs}, and we recommend the reader to read Section~\ref{sec:graphs} first.

Fix a finite $n\geq 2$ (we will handle the case $n=\omega$ at the very end) and a finite $n$-partite tournament $\str A$ with parts
$A_1,A_2,\ldots,A_n$. We will give an explicit construction of an $n$-partite tournament $\str B$ which is an EPPA-witness for $\str A$.

Without loss of generality we can assume the following
\begin{enumerate}
\item $A=\{1,2,\ldots, k\}$,
\item for every $x\in A_i$ and every $y\in A_j$ it holds that $x<y$ whenever $i<j$, and 
\item $|A_1|=|A_2|=\cdots=|A_n|$ (this follows from the fact that if $\str D$ is an EPPA-witness for $\str C$ then it is also an EPPA-witness for all subgraphs of $\str C$).
\end{enumerate}

\subsection{Witness construction}
Given vertex $x\in A_i$ for some $1\leq i\leq n$, we put $N(x) = A\setminus A_i$, and we call every function $\chi\colon N(x)\to \mathbb Z_2$ a \emph{valuation function for $x$}. Note that operations with valuation functions will be performed in $\mathbb Z_2$.

The structure $\str B$ is constructed as follows:
\begin{enumerate}
 \item The vertex set $B$ consists of all pairs $(x,\chi)$ where $x\in A$ and $\chi\colon N(x)\to \mathbb Z_2$ is a valuation function for $x$, and
 \item $(x,\chi)$ and $(x',\chi')$ are adjacent if and only if $x$ and $x'$ belong to different parts of $\str A$. The edge
is oriented from $(x,\chi)$ to $(x',\chi')$ if and only if one of the following is satisfied:
\begin{itemize}
\item $x > x', \text{ and }\chi(x') + \chi'(x) = 1$, or
\item $x < x', \text{ and }\chi(x') + \chi'(x) = 0$.
\end{itemize}
Otherwise the edge is oriented from $(x',\chi')$ to $(x,\chi)$.
\end{enumerate}
It is easy to observe that $\str B$ is an $n$-partite tournament with parts $B_i=\{(x,\chi) \in B : x\in A_i\}$.

Next we construct an embedding $\psi\colon \str A\to \str B$, putting $\psi(x) = (x,\chi_x)$ where $\chi_x\colon N(x) \to \mathbb Z_2$ satisfies
$$\chi_x(y)=
\begin{cases}
1 & \text{ if } y<x \text{ and there is an edge directed from $x$ to $y$ in $\str A$}\\
0 & \text{ otherwise}.
\end{cases}$$
It is easy to verify that $\psi$ is indeed an embedding of $\str A$ into $\str B$.
Put $\str A' = \psi(\str A)$. This is the copy of $\str A$ in $\str B$ whose partial automorphisms we will extend.

\subsection{Automorphisms of $\str B$}
We now define two families of automorphisms of $\str B$, denoted by $\theta_\pi$ and $\theta_{u,v}$ respectively, which we will later use to extend partial automorphisms.

Let $\pi\colon A\to A$ be a \emph{part-preserving} bijection (that is, for every $x,y\in A$, if there is $i$ such that $x,y\in A_i$ then there is $j$ such that $\pi(x),\pi(y)\in A_j$). Define $\theta_\pi\colon B\to B$ such that $\theta_\pi((x,\chi)) = (\pi(x),\chi^\pi)$ where $\chi^\pi\colon N(\pi(x))\to \mathbb Z_2$ is the valuation function for $\pi(x)$ satisfying
$$\chi^\pi(\pi(y)) = \begin{cases}
1+\chi(y) & \text{ if } x < y \text{ and } \pi(x) > \pi(y)\\
\chi(y) & \text{ otherwise}.
\end{cases}$$
\begin{lemma}
For every part-preserving bijection $\pi\colon A\to A$, $\theta_\pi$ is an automorphism of $\str B$.
\end{lemma}
\begin{proof}
First observe that $\theta_\pi$ is injective (and hence bijective as $B$ is finite): If $\theta_\pi((x,\chi)) = \theta_\pi((y,\xi))$ then we know that $\pi(x) = \pi(y)$ (and so $x=y$) and $\chi^\pi = \xi^\pi$. It is easy to verify that $\chi^\pi = \xi^\pi$ implies $\chi = \xi$, and hence $\theta_\pi$ is indeed a bijection. Also note that $\theta_\pi$ preserves parts, and hence non-edges, of $\str B$. Let $(x,\chi)$ and $(y,\xi)$ be vertices of $\str B$ such that there is an edge from $(x,\chi)$ to $(y,\xi)$ (note that this implies $x\neq y$). Put
$$V = \chi(y) + \xi(x) + \chi^\pi(\pi(y)) + \xi^\pi(\pi(x)).$$
To see that $\theta_\pi$ preserves the direction of the edge we need to prove that $V=1$ if and only if $\pi$ is decreasing on the pair $\{x,y\}$. If $\pi$ is increasing on the pair then $\chi^\pi(y) = \chi(y)$ and $\xi^\pi(x) = \xi(x)$, hence $V = 0$. If $\pi$ is decreasing then exactly one of $\chi^\pi(y) = 1+ \chi(y)$ and $\xi^\pi(x) = 1+\xi(x)$ holds (depending on whether $x<y$ or $y<x$), and thus $V=1$.
\end{proof}

Let $u < v\in A$ be vertices from different parts of $\str A$. Define $\theta_{u,v}\colon B\to B$ such that $\theta_{u,v}((x,\chi)) = (x,\chi^{u,v})$ where where $\chi^{u,v}\colon N(x)\to \mathbb Z_2$ is the valuation function for $x$ satisfying
$$\chi^{u,v}(y) = \begin{cases}
1+\chi(y) & \text{ if } \{x,y\} = \{u,v\}\\
\chi(y) & \text{ otherwise}.
\end{cases}$$
\begin{lemma}
For every pair of vertices $u < v\in A$ from different parts of $\str A$, $\theta_{u,v}$ is an automorphism of $\str B$.
\end{lemma}
\begin{proof}
Again, $\theta_{u,v}$ clearly preserves parts, and hence non-edges, of $\str B$. Let $(x,\chi)$ and $(y,\xi)$ be vertices of $\str B$ such that there is an edge from $(x,\chi)$ to $(y,\xi)$. If $\{x,y\} \neq \{u,v\}$ then 
$$\chi^{u,v}(y) + \xi^{u,v}(x) = \chi(y) + \xi(x),$$
and if $\{x,y\} = \{u,v\}$ then
$$\chi^{u,v}(y) + \xi^{u,v}(x) = 1+\chi(y) + 1+\xi(x) = \chi(y) + \xi(x).$$
In both cases the sums of valuation functions, and hence directions of edges, are preserved by $\theta_{u,v}$.
\end{proof}
\begin{observation}\label{obs:npart:com}
For every quadruple $u,v,w,x\in A$ for which $\theta_{u,v}$ and $\theta_{w,x}$ are defined it holds that $\theta_{u,v}\theta_{w,x} = \theta_{w,x}\theta_{u,v}$. \qed
\end{observation}
\begin{observation}\label{obs:npart:com2}
For every part-preserving bijection $\pi\colon A\to A$ and for every $u<v\in A$ for which $\theta_{u,v}$ is defined it holds that $\theta_\pi\theta_{u,v} = \theta_{\pi(u),\pi(v)}\theta_\pi$. \qed
\end{observation}

\subsection{Extending partial automorphisms}
We now show that $\str B$ extends all partial automorphisms of $\str A'$. In fact, this will be almost verbatim the same as in Section~\ref{sec:graphs}. Fix a partial automorphism $\varphi\colon \str A'\to \str A'$. Looking at the first coordinates, it induces a partial permutation of $A$ which is part-preserving. Extend it to a part-preserving bijection $\hat\varphi\colon A\to A$ (this is possible since all parts have the same size).

\begin{observation}\label{obs:npart:proj_fine}
For every $(x,\chi_x) \in \dom(\varphi)$ we have that if $\varphi((x,\chi_x)) = (y,\chi_y)$ then $\theta_{\hat\varphi}((x,\chi_x)) = (y,\chi)$ for some $\chi\colon N(y)\to\mathbb Z_2$. \qed
\end{observation}
In the following we will show that one can compose $\theta_{\hat\varphi}$ with a suitable set of automorphisms $\theta_{u,v}$ to extend $\varphi$. Note that if we believe that this is possible then there is an obvious canonical way of doing it: Given $(x,\chi_x) \in \dom(\varphi)$, with $\varphi((x,\chi_x)) = (y,\chi_y)$ and $\theta_{\hat\varphi}((x,\chi_x)) = (y,\chi)$, we need to fix those entries $z$ for which $\chi_y(z) \neq \chi(z)$.

First, we define a set $F$ consisting of all pairs $\{u,v\}$ satisfying the following:
\begin{enumerate}
  \item $u < v \in A$ and they are from different parts of $\str A$,
  \item either $(u,\chi_u) \in \dom(\varphi)$ or $(v,\chi_v)\in \dom(\varphi)$ (or both),
  \item\label{cond3} if $(u,\chi_u) \in \dom(\varphi)$, $\varphi((u,\chi_u)) = (x,\chi_x)$, and $\theta_{\hat\varphi}((u,\chi_u)) = (x,\chi)$ then $\chi(\hat\varphi(v)) \neq \chi_x(\hat\varphi(v))$, and
  \item\label{cond4} if $(v,\chi_v) \in \dom(\varphi)$, $\varphi((v,\chi_v)) = (y,\chi_y)$, and $\theta_{\hat\varphi}((v,\chi_v)) = (y,\xi)$ then $\xi(\hat\varphi(u)) \neq \chi_y(\hat\varphi(u))$.
\end{enumerate}
\begin{observation}\label{obs:Fokay}
If $u < v$ are from different parts of $A$, and $(u,\chi_u), (v,\chi_v) \in \dom(\varphi)$ then~\ref{cond3} is satisfied if and only if~\ref{cond4} is.
\end{observation}
\begin{proof}
This follows from the fact that $\theta_{\hat\varphi}$ is an automorphism of $\str B$ and the fact that $\varphi$ is a partial automorphism (and thus it preserves the orientation of the edge $(u,\chi_u), (v,\chi_v)$): Given the order of the first coordinates (which is fixed) and the direction of the edge (which is fixed), knowing $\chi(\hat\varphi(v))$ or $\chi_x(\hat\varphi(v))$ respectively determines $\xi(\hat\varphi(u))$ resp. $\chi_y(\hat\varphi(u))$ and vice versa.
\end{proof}

Let $\theta_F$ be the composition of all $\theta_{u,v}$ for $\{u,v\}\in F$ (by Observation~\ref{obs:npart:com}, $\theta_F$ does not depend on the order of the composition). Put $\theta = \theta_{\hat\varphi}\theta_F$.

\begin{prop}
$\theta$ is an automorphism of $\str B$ extending $\varphi$.
\end{prop}
\begin{proof}
The fact that $\theta$ is an automorphism is clear as it is the composition of several automorphisms of $\str B$. Observation~\ref{obs:npart:proj_fine} gives us that $\theta$ and $\varphi$ agree on the first coordinate. From the definition of $F$ and Observation~\ref{obs:Fokay} it follows that if $(u,\chi_u)\in \dom(\varphi)$, $\varphi((u,\chi_u)) = (x,\chi_x)$, and $\theta_{\hat\varphi}((u,\chi_u)) = (x,\chi)$, then $\{u,v\} \in F$ if and only if $\chi_x(v) \neq \chi(v)$. Consequently, $\theta$ indeed extends $\varphi$.
\end{proof}

We are now ready to prove Theorem~\ref{thm:partite}:

\begin{proof}[Proof of Theorem~\ref{thm:partite}]
Given $n \in \{2,\ldots,\omega\}$, let $\str A$ be a finite $n$-partite tournament with $n'$ non-empty parts. Use the construction from this section to construct a finite $n'$-partite tournament $\str B$ with $n'$ non-empty parts. We have proved that $\str B$ is an $n$-partite tournament and that it is an EPPA-witness for $\str A' = \psi(\str A)$. Clearly, by taking an isomorphism, one gets an EPPA-witness for $\str A$.
\end{proof}

\section{Ample generics for $n$-partite tournaments}\label{sec:ample:npart}
The aim of this section is to prove Theorem~\ref{thm:npart_ample}. In order to do so, we first need to review some facts about ample generics.

\subsection{Background}\label{sec:ample:background}
Recall Definition~\ref{defn:ample}. The $n=1$ case was first studied by Truss~\cite{Truss1992}. Later, Hodges, Hodkinson, Lascar, and Shelah~\cite{hodges1993b} proved that the random graph has ample generics (and used Hrushovski's theorem about EPPA for graphs in their proof), and, as a consequence, the small index property. Their methods were abstracted by Kechris and Rosendal~\cite{Kechris2007}:

\begin{definition}
Let $L$ be a language, let $\mathcal C$ be a class of finite $L$-structures and let $n\geq 1$ be an integer. An \emph{$n$-system over $\mathcal C$} is a tuple $(\str A, p_1,\ldots,p_n)$, where $\str A\in \mathcal C$ and $p_1,\ldots,p_n$ are partial automorphisms of $\str A$. We denote by $\mathcal C^n$ the class of all $n$-systems over $\mathcal C$.

If $P=(\str A,p_1,\ldots,p_n)$ and $Q=(\str B,q_1,\ldots,q_n)$ are both $n$-systems over $\mathcal C$ and $f\colon \str A\to\str B$ is an embedding of $L$-structures, we say that $f$ is an \emph{embedding of $n$-systems $P\to Q$} if for every $1\leq i \leq n$ it holds that $f\circ p_i \subseteq q_i\circ f$ (in particular, $f(\dom(p_i))\subseteq \dom(q_i)$ and $f(\range(p_i))\subseteq \range(q_i)$).
\end{definition}

\begin{definition}
Let $L$ be a language, let $\mathcal C$ be a class of finite $L$-structures and let $n\geq 1$ be an integer. We say that $\mathcal C^n$ has the \emph{joint embedding property} if for every $P,Q\in\mathcal C^n$ there exists $S\in \mathcal C^n$ with embeddings of $n$-systems $f\colon P\to S$ and $g\colon Q\to S$. We say that $\mathcal C^n$ has the \emph{weak amalgamation property} if for every $T\in \mathcal C^n$ there exists $\hat{T}\in \mathcal C^n$ and an embedding of $n$-systems $\iota\colon T\to\hat{T}$ such that for every pair of $n$-systems $P,Q\in \mathcal C^n$ and embeddings of $n$-systems $\alpha_1\colon \hat{T}\to P$ and $\alpha_2\colon \hat{T}\to Q$ there exists $S\in \mathcal C^n$ with embeddings of $n$-systems $\beta_1\colon P\to S$ and $\beta_2\colon Q\to S$ such that $\beta_1\alpha_1\iota = \beta_2\alpha_2\iota$.
\end{definition}

\begin{theorem}[Kechris--Rosendal~\cite{Kechris2007}]\label{thm:kr}
Let $\str M$ be a countable locally finite homogeneous structure. For every $n\geq 1$ it holds that $\str M$ has $n$-generic automorphisms if and only if $\age(\str M)^n$ has the joint embedding property and the weak amalgamation property.
\end{theorem}

In order to explain the connection between EPPA and ample generics, we need one more standard definition (which we give in a slightly more general way as it is going to be convenient later when we apply it):

\begin{definition}
Let $L$ be a language, let $\mathcal C$ be a class of finite $L$-structures, and let $\mathcal C' \subseteq \mathcal C$. We say that $\mathcal C$ has the \emph{amalgamation property with automorphisms} (abbreviated as \emph{APA}) \emph{over $\mathcal C'$} if for every $\str{A}\in \mathcal C'$, every $\str{B}_1,\str{B}_2\in \mathcal C$, and embeddings $\alpha_1\colon\str{A}\to\str{B}_1$, $\alpha_2\colon\str{A}\to\str{B}_2$ there exists $\str{C}\in \mathcal C$ with embeddings $\beta_1\colon\str{B}_1 \to \str{C}$ and $\beta_2\colon\str{B}_2\to\str{C}$ such that $\beta_1\circ\alpha_1 = \beta_2\circ\alpha_2$ (i.e. $\str C$ is an amalgamation of $\str B_1$ and $\str B_2$ over $\str A$ with respect to $\alpha_1$ and $\alpha_2$), and moreover for every $f\in \Aut(\str B_1)$ and $g\in \Aut(\str B_2)$ there is $h\in \Aut(\str C)$ which extends $\beta_1 f \beta_1^{-1} \cup\beta_2 g \beta_2^{-1}$ provided that $f(\alpha_1(A)) = \alpha_1(A)$, $g(\alpha_2(A)) = \alpha_2(A)$, and $\alpha_1^{-1}f\alpha_1 = \alpha_2^{-1}g\alpha_2$ (that is, $f$ and $g$ agree on the copy of $\str A$ we are amalgamating over). We call such $\str C$ with embeddings $\beta_1$ and $\beta_2$ an \emph{APA-witness} for $\str B_1$ and $\str B_2$ over $\str A$ with respect to $\alpha_1$ and $\alpha_2$. If $\mathcal C' = \mathcal C$ we drop the ``over $\mathcal C'$'' part.
\end{definition}

\begin{prop}[Kechris--Rosendal~\cite{Kechris2007}]\label{prop:ample_apa}
Let $L$ be a language, and let $\mathcal C'\subseteq \mathcal C$ be classes of finite $L$-structures. Suppose that $\mathcal C$ has APA over $\mathcal C'$ and that for every $\str A\in \mathcal C$ there is $\str B \in \mathcal C'$ which is an EPPA-witness for $\str A$. Then $\mathcal C^n$ has the weak amalgamation property for every $n\geq 1$. If $\mathcal C'$ contains the empty structure then $\mathcal C^n$ also has the joint embedding property for every $n\geq 1$.
\end{prop}
\begin{proof}
Fix $n\geq 1$. If $S = (\str S,s_1,\ldots,s_n) \in \mathcal C^n$ is an $n$-system, we denote by $\hat{S} = (\hat{\str S}, \hat{s}_1,\ldots,\hat{s}_n)\in \mathcal C^n$ the $n$-system where $\hat{\str S} \in \mathcal C'$ is an EPPA-witness for $\str S$ and for every $1\leq i\leq n$ it holds that $\hat{s}_i$ is an automorphism of $\hat{\str S}$ extending $s_i$.

We now prove that $\mathcal C^n$ has the weak amalgamation property. Towards that, fix some $T = (\str T, t_1,\ldots,t_n)\in \mathcal C^n$. Let $P=(\str P,p_1,\ldots,p_n),Q=(\str Q,q_1,\ldots,q_n)\in \mathcal C^n$ be arbitrary $n$-systems with embeddings $\alpha_1\colon \hat{T}\to P$ and $\alpha_2\colon \hat{T}\to Q$.

Use APA for $\mathcal C$ over $\mathcal C'$ to get $\str S\in \mathcal C$ and embeddings $\beta_1\colon\hat{\str P} \to \str S$ and $\beta_2\colon\hat{\str Q}\to\str S$ such that $\str S$ with $\beta_1$ and $\beta_2$ form an APA-witness for $\hat{\str P}$ and $\hat{\str Q}$ over $\hat{\str{T}}$ with respect to $\alpha_1$ and $\alpha_2$. Clearly, $S$ is the desired $n$-system witnessing the weak amalgamation property for $P$, $Q$ and $T$.

The joint embedding property is simply the amalgamation property over the empty $n$-system.
\end{proof}

\subsection{Proof of Theorem~\ref{thm:npart_ample}}

\begin{lemma}\label{lem:omegapart_apa}
$\omega$-partite tournaments have APA.
\end{lemma}
\begin{proof}
Let $\str A$ be an $\omega$-partite tournament, and let $\beta_1\colon \str A\to \str B_1$ and $\beta_2\colon \str A\to \str B_2$ be embeddings. Without loss of generality we can assume that both $\beta_1$ and $\beta_2$ are inclusions and that $B_1\cap B_2 = A$.

Put $C = B_1\cup B_2$. We first define a partition of $C$ such that $x,y\in C$ are in the same part of $C$ if either $x,y$ are in the same part of $\str B_1$, or they are in the same part of $\str B_2$, or there is $z\in A$ such that $x,z$ are in the same part of $\str B_1$ and $z,y$ are in the same part of $\str B_2$.

We will define an $\omega$-partite tournaments $\str C$ with vertex set $C$ which is an amalgamation of $\str B_1$ and $\str B_2$ over $\str A$. Given $u,v$ from different parts of $C$, if $u,v\in B_1$ or $u,v\in B_2$, we orient the edge $uv$ according to $\str B_1$ resp. $\str B_2$. Otherwise we have without loss of generality $u\in B_1\setminus B_2$ and $v\in B_2\setminus B_1$ and we orient the edge from $u$ to $v$.

Clearly, $\str C$ is an $\omega$-partite tournament, and an amalgamation of $\str B_1$ and $\str B_2$ over $\str A$. Now let $f_1\colon \str B_1\to \str B_1$ and $f_2\colon \str B_2\to \str B_2$ be automorphisms such that $f_1\restriction_A = f_2\restriction_A$. Put $f = f_1\cup f_2$, and observe that $f$ is a bijection $C\to C$ which preserves the partition of $\str C$: The fact that it is a bijection follows from the definition of $C$ and from $f_1\restriction_A = f_2\restriction_A$. It preserves the parts of $\str C$ because every part of $\str C$ is either a part of $\str B_1$ or $\str B_2$ (in which case it is preserved by $f_1$ or $f_2$ respectively), or the union of one part of $\str B_1$ and one part of $\str B_2$ which have non-empty intersection in $\str A$ (and in this case $f_1$ and $f_2$ agree on this intersection).

If $u,v\in B_1$ or $u,v\in B_2$ then $f$ preserves the direction of the edge between them (or the non-existence thereof). If $u\in B_1\setminus B_2$ and $v\in B_2\setminus B_1$ then either they are in the same part of $C$ (and so are $f(u)$ and $f(v)$), or they are in different parts, but then the edges are oriented from $u$ to $v$ and from $f(u)$ to $f(v)$. Hence $f$ is an automorphism of $\str C$.
\end{proof}

Now we can prove Theorem~\ref{thm:npart_ample}.
\begin{proof}[Proof of Theorem~\ref{thm:npart_ample}]
First consider the case $n=\omega$. Then Theorem~\ref{thm:kr} and Proposition~\ref{prop:ample_apa} together with Theorem~\ref{thm:partite} and Lemma~\ref{lem:omegapart_apa} give the desired result.

If $n$ is finite, let $\str B_0$ be the oriented 4-cycle, which is a 2-partite tournament. Let $\str B$ be the $n$-partite tournament obtained by adding $n-2$ one-vertex parts to $\str B_0$ such that all edges between $u\in B_0$ and $v\in B\setminus B_0$ are oriented from $u$ to $v$, and the edges between vertices of $B\setminus B_0$ are oriented arbitrarily. Let $f$ be the automorphism of $\str B$ fixing $B\setminus B_0$ pointwise and rotating $\str B_0$ by one vertex (so, in particular, it exchanges the two parts of $\str B_0$). Then, clearly, the 1-systems $(\str B, \mathrm{id})$ and $(\str B, f)$ have no joint embedding: Any 1-system to which $(\str B, \mathrm{id})$ embeds fixes all parts, while $(\str B, f)$ does not fix all parts. Consequently, $n$-partite tournaments do not have 1-generic automorphisms by Theorem~\ref{thm:kr}.

By Theorem~1.1 of~\cite{Kechris2007}, the joint embedding property for 1-systems over $\mathcal C$ is equivalent to the automorphism group of the \Fraisse{} limit of $\mathcal C$ having a dense conjugacy class from which the last part of the theorem follows.
\end{proof}

\begin{remark}
Let us remark that for every $n\geq 2$, $n$-partite tournaments have APA over structures which have all $n$ parts non-empty. Indeed, let $\str A$, $\str B_1$ and $\str B_2$ be $n$-partite tournaments such that $\str A$ is a substructure of both $\str B_1$ and $\str B_2$, $B_1\cap B_2 = A$, and $\str A$ has all $n$ parts non-empty. Observe that the structure $\str C$ produced by Lemma~\ref{lem:omegapart_apa} is in fact an $n$-partite tournament (because it merges each part of $\str B_1$ with exactly one part of $\str B_2)$ which proves the claim.

As we have seen in the proof of Theorem~\ref{thm:npart_ample}, they do not have APA over the empty structure. The same situation happens for equivalences with $n$ equivalence classes and it is a consequence of these classes having non-trivial algebraic closure of the empty set after eliminating imaginaries (namely, the representatives of the equivalence classes are in the algebraic closure). If one only considers automorphisms which agree on $\mathrm{acl}(\emptyset)$ then it is possible to do the whole construction and obtain generic automorphisms with a fixed action on $\mathrm{acl}(\emptyset)$.
\end{remark}

\section{EPPA for semigeneric tournaments}\label{sec:semig}
In this section we will prove Theorem~\ref{thm:semigeneric}. The construction is slightly more involved than the one for $n$-partite tournaments, but there is a natural example showing why this is necessary. In order to present it, we need to start with the following observation:

\begin{observation}\label{obs:semig_equiv_weak}
Let $\str A$ be a semigeneric tournament and let $X$ and $Y$ be a pair of parts of $\str A$. For every $a,b\in X$, if there is $c\in Y$ such that the edges $ac$ and $bc$ have the same orientation then for every $d\in Y$ the edges $ad$ and $bd$ have the same orientation.
\end{observation}
\begin{proof}
Since $\str A$ is semigeneric, we have that the number of edges going from $\{a,b\}$ to $\{c,d\}$ is even. We know that the edges $ac$ and $bc$ go in the same direction, and hence $ad$ and $bd$ need to as well.
\end{proof}

\begin{example}\label{ex:semig_bad}
Let $I$ and $O$ be disjoint sets not containing $v$ such that $\lvert I\rvert = \lvert O\rvert = n$ for some $n$. Let $\str A$ be the 2-partite semigeneric tournament with parts $A_1$ and $A_2$ such that $A_1 = I\cup O$, $A_2 = \{v\}$, and the edges are oriented from $v$ to members $I$ and from members of $O$ to $v$. Assume that $\str B$ is a semigeneric tournament which is an EPPA-witness for $\str A$. We will show that for every $S\subseteq I$ there is a vertex $v_S\in B$ with edges oriented from $v_S$ to $S$ and from $I\setminus S$ to $v_S$. Note that only $v_S$ and $v_{I\setminus S}$ can be in the same part of $\str B$ -- this follows from Observation~\ref{obs:semig_equiv_weak}. Consequently, this implies that $\str B$ needs to have at least $2^{n-1}$ parts.

To see this, fix $S\subseteq I$, and let $\varphi\colon A_1\to A_1$ be some permutation satisfying that $\varphi(i) = i$ for every $i\in S$ and $\varphi(i)\in O$ for every $i\in I\setminus S$ (it certainly exists). Notice that $\varphi$ is a partial automorphism of $\str A$, and as such, extends to an automorphism $\theta\colon \str B\to \str B$. We can then put $v_S = \theta(v)$.
\end{example}

\medskip

We now continue with the construction. For the rest of the section, fix a finite semigeneric tournament $\str A$ with vertex set $A$ and with $k$ non-empty parts $A_1, \ldots, A_k$. Put $P=\{1,\ldots,k\}$ and put $n = \lvert A\rvert$. We will give an explicit construction of a semigeneric tournament $\str B$ which is an EPPA-witness for $\str A$. Without loss of generality we can assume that $ \lvert A_1\rvert =  \lvert A_2\rvert = \cdots = \lvert A_k\rvert$, and that $A = \{1,\ldots,n\}$. (We will only use that $A$ is linearly ordered by $\leq$.)

\subsection{Witness construction}\label{sec:semig:const}
Let $Q$ be the set of all pairs $(A_i, h)$, where $i\in P$ and $h\colon A\to \mathbb Z_2$. The vertex set $B$ of $\str B$ will consist of all triples $((A_i, h), v, \chi)$, where $(A_i, h)\in Q$, $v\in A_i$ and $\chi\colon Q\to \mathbb Z_2$.

The parts of $\str B$ will be indexed by elements of $Q$ and for this reason we call functions $h\colon A\to \mathbb Z_2$ \emph{part valuation functions}. Their role is to``multiply'' the parts of $\str A$ and eventually allow us to extend partial automorphisms such as the one from Example~\ref{ex:semig_bad}. On the other hand, functions $\chi\colon Q\to \mathbb Z_2$ will be called \emph{vertex valuation functions} and they will play a very similar role as valuation functions in Section~\ref{sec:npart}. Observation~\ref{obs:semig_equiv_weak} suggests that very little information is needed to determine the directions of all edges between a pair of parts, and for this reasons the domain of vertex valuation functions is just $Q$ instead of, say, $A\setminus A_i$ like in Section~\ref{sec:npart}. 

We will denote part valuation functions by letters $g$ an $h$ and their decorations such as $g'$ or $h_0$, and vertex valuation functions by Greek letters $\chi$, $\xi$, and $\zeta$, and their decorations such as $\chi_x$ or $\zeta'$. For the rest of this section let $h_0\colon A\to \mathbb Z_2$ be the constant 0 part valuation function.

Let $\triangleleft$ be the order on part valuation functions defined by $g\triangleleft h$ if and only if $g\neq h$ and $g(v) = 0$ for $v\in A$ being the $\leq$-least vertex for which $g(v)\neq h(v)$. Note that $h_0$ is the minimum of $\triangleleft$. Let $\prec$ be the linear order on $Q$ defined by $(A_i, h) \prec (A_j, h')$ if and only if either $i < j$ or $i=j$ and $h \triangleleft h'$.

Given $\str x = ((A_i, h), x, \chi)\in B$ and $\str y = ((A_p, g), y, \xi)\in B$ such that $(A_i,h)\neq (A_p,g)$, we define (recall that all operations are performed in $\mathbb Z_2$)
$$V(\str x,\str y) = \chi((A_p, g)) + h(y) + \xi((A_i, h)) + g(x).$$

We can now define the edges of $\str B$: We connect $\str v=((A_i, h), v, \chi)$ with $\str v'=((A_{i'}, h'), v', \chi')$ in $\str B$ if and only if $(A_i, h)\neq(A_{i'}, h')$. The edge is oriented from $\str v$ to $\str v'$ if and only if one of the following is satisfied (otherwise it is oriented from $\str v'$ to $\str v$).
\begin{enumerate}
\item $(A_i, h)\prec(A_{i'}, h')$ and $V(\str v,\str v') = 0$, or
\item $(A_i, h)\succ(A_{i'}, h')$ and $V(\str v,\str v') = 1$.
\end{enumerate}

\begin{lemma}
$\str B$ is a semigeneric tournament.
\end{lemma}
\begin{proof}
It is easy to see that $\str B$ is a $\lvert Q\rvert$-partite tournament with parts given by the first coordinate. To see that $\str B$ is a semigeneric tournament, fix two parts $(A_i, h)\prec (A_{i'}, h')$ and inside each one fix two vertices $((A_i, h), v_1, \chi_1),\allowbreak ((A_i, h), v_2, \chi_2)$ resp. $((A_{i'}, h'), v_1', \chi_1'), ((A_{i'}, h'), v_2', \chi_2')$.

Note that changing $h'(v_1)$ from 0 to 1 or vice-versa changes the direction of two or four edges on these four vertices (based on whether $v_1=v_2$), so in particular it preserves parities and thus we can assume that $h'(v_1)=0$. The same argument gives us that, without loss of generality, $h'(v_1) = h'(v_2)=h(v_1')=h(v_2')=0$.

Note also that changing $\chi_1((A_{i'}, h'))$ changes the directions of precisely two edges and therefore we can assume that
$$\chi_1((A_{i'}, h'))=\chi_2((A_{i'}, h'))=\chi_1'((A_{i}, h))=\chi_2'((A_{i}, h))=0.$$
But then all four edges go in one direction and thus the parity condition is satisfied and $\str B$ is indeed a semigeneric tournament.
\end{proof}

Next we construct an embedding $\psi\colon \str A\to \str B$. For every $i\in P$, fix a representative $y^i\in A_i$. For every $p\in P$ and every $x\in A_p$, put $\psi(x) = ((A_p, h_0), x, \chi_x)$, where $\chi_x\colon Q\to \mathbb Z_2$ satisfies
$$
\chi_x((A_i, h))=
\begin{cases}
  0 & \text{if }i = p,\\
  0 & \text{if }i > p \text{ and there is an edge from $x$ to $y^i$},\\
  1 & \text{if }i > p \text{ and there is an edge from $y^i$ to $x$},\\
  0 & \text{if }i < p \text{ and the orientations of edges $y^ix$ and  $y^iy^{p}$ agree},\\
  1 & \text{if }i < p \text{ and the orientations of edges $y^ix$ and  $y^iy^{p}$ disagree}.
\end{cases}
$$
\begin{lemma}
$\psi$ is an embedding $\str A\to \str B$.
\end{lemma}
\begin{proof}
Fix arbitrary $u,v\in A$. If they are in the same part then clearly there is no edge between $u$ and $v$ in $\str A$ and there is no edge between $\psi(u)$ and $\psi(v)$ in $\str B$. So we can assume that $u\in A_i$, $v\in A_p$,

Note that, in any semigeneric tournament, if we have two vertices from one part and two vertices from some other part, the directions of any three edges between them determine the direction of the fourth edge. Consequently, if $\psi$ preserves the directions of edges $y^iy^p$, $y^iv$ and $uy^p$ then it also preserves the direction of edge $uv$. Thus, we can without loss of generality assume that $v = y^p$.

First assume that $i < p$. From the construction we have that $\chi_v((A_i,h_0)) = 0$. We also know that $\chi_u((A_p,h_0)) = 0$ if and only if there is an edge from $u$ to $v=y^p$. Consequently, $V(\psi(u),\psi(v)) = \chi_u((A_p,h_0))$ and the definition of the edges of $\str B$ implies that $\psi$ preserves the direction of the edge $uv$.

So $i > p$. In this case $V(\psi(u),\psi(v)) = 0$ if and only if exactly one of the following holds: Either there is an edge from $y^i$ to $y^p$ and the orientations of edges $y^iy^p$ and $uy^p$ disagree, or there is an edge from $y^p$ to $y^i$ and the orientations of edges $y^iy^p$ and $uy^p$ agree. Both of these are equivalent to the existence of an edge from $y^p$ to $u$, hence $\psi$ again preserves the direction of the edge $uv$, and so $\psi$ is indeed an embedding $\str A\to \str B$.
\end{proof}
Put $\str A' = \psi(\str A)$. This will be the copy of $\str A$ in $\str B$ whose automorphisms we will extend.

\subsection{Automorphisms of $\str B$}
We will now define three families of automorphisms of $\str B$, denoted by $\theta_\pi$, $\theta_{a,b}$, and $\theta_{a,v}$ respectively, which we will later use to extend partial automorphisms of $\str A'$. However, before we do it, we state an easy observation which will be helpful in proving that what we defined are indeed automorphisms.

Given $\omega$-partite tournaments $\str C$ and $\str D$ and an injection $f\colon C\to D$, we say that $f$ is \emph{part-preserving} if for every $x,y\in C$, $x$ and $y$ are connected by an edge of $\str C$ if and only if $f(x)$ and $f(y)$ are connected by an edge of $\str D$. Given a part-preserving bijection $f\colon \str B\to \str B$, we denote by $f_Q$ the action of $f$ on $Q$ given by projecting to the first coordinate (note that it is well-defined since $f$ is part-preserving), and we say, for $\str x = ((A_i, h), x, \chi), \str y = ((A_j, g), y, \xi)\in B$, that $f$ \emph{flips the parts of $\str x$ and $\str y$} if $(A_i,h) \neq (A_j,g)$ and $f_Q$ is decreasing (in the $\preceq$ order) when restricted to $\{(A_i,h), (A_j,g)\}$.

\begin{observation}\label{obs:aut_V}
Let $C\subseteq B$ and let $f\colon C\to B$ be a part-preserving injection. Then $f$ is a partial automorphism of $\str B$ if and only if for every $\str x, \str y\in C$ from different parts we have that $f$ flips the parts of $\str x$ and $\str y$ if and only if $V(\str x,\str y) + V(f(\str x),f(\str y)) = 1$. \qed
\end{observation}

Let $\pi\colon A\to A$ be a part-preserving bijection. Let $\iota\colon P\to P$ be the permutation of $P$ induced by $\pi$. Given $h\colon A\to \mathbb Z_2$, denote by $h^\pi\colon A\to \mathbb Z_2$ the part valuation function satisfying $h^\pi(y) = h(\pi^{-1}(y))$.

Define $\theta_\pi\colon B\to B$ such that $\theta_\pi(((A_i, h), x, \chi)) = ((A_{\iota(i)}, h^\pi), \pi(x), \chi')$ where $\chi'\colon Q\to\mathbb Z_2$ is the vertex valuation function satisfying
$$\chi'((A_{\iota(p)}, g^\pi)) = \begin{cases}
1+\chi((A_p, g)) & \text{ if } (A_i, h) \prec (A_p, g) \text{ and } (A_{\iota(i)}, h^\pi) \succ  (A_{\iota(p)}, g^\pi)\\
\chi((A_p, g)) & \text{ otherwise}.
\end{cases}$$

\begin{lemma}
For every part-preserving bijection $\pi\colon A\to A$ it holds that $\theta_\pi$ is an automorphism of $\str B$.
\end{lemma}
\begin{proof}
Using a similar argument as in Section~\ref{sec:npart} and noting that the map $h\mapsto h^\pi$ is a bijection, it is easy to see that $\theta_\pi$ is a part-preserving bijection. Observe that from the construction of $\chi'$ it follows that $V(\str x,\str y)=V(\theta_\pi(\str x),\theta_\pi(\str y))$ if and only if $\theta_\pi$ does not flip the parts of $\str x$ and $\str y$. Observation~\ref{obs:aut_V} then implies that $\theta_\pi$ is an automorphism.
\end{proof}

Given $a < b\in P$ we define $\theta_{a,b}\colon B\to B$ such that $\theta_{a,b}(((A_i, h), x, \chi)) = ((A_i, h), x, \chi')$ where $\chi'\colon Q\to \mathbb Z_2$ is the vertex valuation function satisfying
$$\chi'((A_{p}, g)) = \begin{cases}
1+\chi((A_p, g)) & \text{ if } \{a,b\} = \{i,p\}\\
\chi((A_p, g)) & \text{ otherwise}.
\end{cases}$$

\begin{lemma}
For every pair $a<b \in P$, $\theta_{a,b}$ is an automorphism of $\str B$.
\end{lemma}
\begin{proof}
First note that $\theta_{a,b} = \theta_{a,b}^{-1}$, and so it is a bijection. Note that $\theta_{a,b}$ sends every part to itself, and so it clearly is part-preserving and does not flip the parts of any vertices of $\str B$. Consequently, by Observation~\ref{obs:aut_V} we only need to prove that for every $\str x,\str y\in B$ we have that $V(\str x,\str y)=V(\theta_{a,b}(\str x),\theta_{a,b}(\str y))$. This follows easily from the construction of $\chi'$ by distinguishing the cases $\{a,b\} = \{i,p\}$ and $\{a,b\} \neq \{i,p\}$.
\end{proof}

So far we have been essentially copying the construction for $n$-partite tournaments. However, in order to address what we have seen in Example~\ref{ex:semig_bad} we need to introduce one extra family of automorphisms which will (unlike $\theta_\pi$'s and $\theta_{a,b}$'s) take advantage of having part valuation functions.

Given $v\in A$ and a part valuation function $h\colon A\to \mathbb Z_2$, let $h^v\colon A\to \mathbb Z_2$ be the part valuation function satisfying
$$h^v(x) = \begin{cases}
1+h(x) & \text{ if } x=v\\
h(x) & \text{ otherwise}.
\end{cases}$$

For every $a\in P$, and $v\in A\setminus A_a$ define $\theta_{a,v}\colon B\to B$ such that $\theta_{a,v}(((A_i, h), x, \chi)) = ((A_i, h'), x, \chi')$ where $h' = h^v$ if $i=a$ and $h'=h$ otherwise, and $\chi'\colon Q\to \mathbb Z_2$ is the vertex valuation function satisfying
$$\chi'((A_{p}, g)) = \begin{cases}
1+\chi((A_p, g^v)) & \text{ if }(a=p=i\land h\triangleleft g^v\land h^v\triangleright g)\\
1+\chi((A_p, g^v)) & \text{ if } \{a,v\} = \{p,x\}\text{ and }a\neq i\\
\chi((A_p, g^v)) & \text{ if } a=p \text{ and }a\neq i\text{ and }x\neq v\\
\chi((A_p, g)) & \text{ otherwise}.
\end{cases}$$
While the conditions may look complicated, they are actually straightforward: In particular, note that we use $\chi((A_p, g^v))$ in the definition of $\chi'$ if and only if $p=a$, and since $g = (g^v)^v$, this is in fact the same situation as when, in the construction of $\theta_\pi$ we define $\chi'((A_{\iota(p)}, g^\pi))$ instead of $\chi'((A_p, g))$ --- as $\theta_{a,v}$ moves some parts, it also needs to act on the vertex valuation functions accordingly.

When moreover $x=v$ then we need to account for us having $g^v(x) = 1+g(x)$ --- while $\theta_{a,b}$ flipped the mutual vertex valuation function values of vertices in both parts, $\theta_{a,v}$ flips the part valuation function value of vertices from part $a$ and the corresponding vertex valuation function value of vertices with $v$ in the second coordinate. Finally, in the case when $a=p=i$ we need to account for $\theta_{a,v}$ flipping the parts of some vertices (similarly as we do for $\theta_\pi$).

\begin{lemma}
For every $a\in P$ and $v\in A\setminus A_a$, the function $\theta_{a,v}$ is an automorphism of $\str B$.
\end{lemma}
\begin{proof}
Once again, one can easily verify that $\theta_{a,v}$ is a bijection which preserves parts, and hence non-edges, of $\str B$. We will now use Observation~\ref{obs:aut_V} to show that $\theta_{a,v}$ is an automorphism. Towards that, let $\str x = ((A_i, h), x, \chi)$ and $\str y = ((A_p, g), y, \xi)$ be vertices from different parts of $\str B$. Put $\theta_{a,v}(\str x) = ((A_i, h'), x, \chi')$ and $\theta_{a,v}(\str y) = ((A_p, g'), y, \xi')$, and put $V = V(\str x,\str y) + V(\theta_{a,v}(\str x),\theta_{a,v}(\str y))$. Then
\begin{align*}
V &= \chi((A_p, g)) + h(y) + \xi((A_i, h)) + g(x)\\
  &+ \chi'((A_p, g')) + h'(y) + \xi'((A_i, h')) + g'(x).
\end{align*}

We first deal with the case $i=p=a$. In this case $h'=h^v$, $g'=g^v$, but $h(y) = h^v(y)$ and $g(x) = g^v(x)$ (as $x,y\in A_a$ and $v\notin A_a$), and thus 
$$V = \chi((A_i, g)) + \xi((A_i, h))  + \chi'((A_i, g^v))  + \xi'((A_i, h^v)).$$
Notice that $\theta_{a,v}$ flips the parts of $\str x$ and $\str y$ if and only if the map $f\mapsto f^v$ is decreasing (with respect to $\triangleleft$) on $\{h,g\}$. If $f\mapsto f^v$ is increasing on $\{h,g\}$ then $\chi((A_i, g))=\chi'((A_i, g^v))$ and $\xi((A_i, h))=\xi'((A_i, h^v))$, so $V=0$. If $f\mapsto f^v$ is decreasing on $\{h,g\}$ then exactly one of the equalities fails (based on the $\triangleleft$-order of $g$ and $h$), and so $V = 1$.

So either $i\neq a$ or $p\neq a$. In this case $\theta_{a,v}$ does not flip the parts of $\str x$ and $\str y$, and so we need to prove that $V=0$. We now have four cases depending on which subset of the equations $\{a,v\} = \{i,y\}$ and $\{a,v\} = \{p,x\}$ holds. If none of them hold then $h'(y) = h(y)$, $g'(x)=g(x)$, $\chi'((A_p, g'))=\chi((A_p, g))$, and $\xi'((A_i, h'))=\xi((A_i, h))$, hence $V=0$. If one of them holds then exactly two of the four equalities fail, hence still $V=0$. Both of them actually cannot hold at the same time, but if they did then all four equalities would fail and we would have $V=0$ nonetheless.
\end{proof}

\begin{observation}\label{obs:semig:com1}
Given $u\in A$, $a,b,c,d\in P$, and a part-preserving bijection $\pi\colon A\to A$ then the following equalities hold whenever their constituents are all defined:
\begin{enumerate}
\item $\theta_{a,b}\theta_{c,d} = \theta_{c,d}\theta_{a,b}$
\item $\theta_{a,b}\theta_{c,u} = \theta_{c,u}\theta_{a,b}$\qed
\end{enumerate}
\end{observation}

It is true that for every $u,v\in A$ and every $a,b\in P$ we have that $\theta_{a,u}\theta_{b,v} = \theta_{b,v}\theta_{a,u}$ whenever $\theta_{a,u}$ and $\theta_{b,v}$ are defined. Proving it takes, however, some work, and we are not going to need the full strength of this. Instead, we are only going to need the following easy observation:
\begin{observation}\label{obs:semig:com2}
Given $a,b\in P$, $u\in A\setminus A_a$, $v\in A\setminus A_b$,
and $\str x = ((A_i, h), x, \chi) \in B$, if $i\notin \{a,b\}$ then
$\theta_{a,u}\circ \theta_{b,v}(\str x) = \theta_{b,v}\circ \theta_{a,u}(\str x)$.\qed
\end{observation}

\subsection{Extending partial automorphisms}
We now show that $\str B$ extends all partial automorphisms of $\str A'$. We will proceed similarly to Section~\ref{sec:npart}. However, Example~\ref{ex:semig_bad} shows that we are going to have to do more work. Notice that it was crucial in Example~\ref{ex:semig_bad} that no vertex from $A_2$ was in the domain of the partial automorphism: As soon as at least one vertex from some part is in the domain of the partial automorphism then Observation~\ref{obs:semig_equiv_weak} ensures that the partial automorphism ``is tame'' on this part (respects the equivalences given by all the other parts used by the partial automorphisms, cf. Observation~\ref{obs:semig_equivalences_pair}).

Hence, similarly to Section~\ref{sec:npart}, we will start by approximating the partial automorphism $\varphi$ by some $\theta_{\hat\varphi}$ which will already agree with the partial automorphism on all coordinates except for the vertex valuation functions. We will subsequently use the $\theta_{a,b}$ and $\theta_{a,v}$ automorphisms: We will use the $\theta_{a,b}$ automorphisms to fix those coordinates of vertex valuation functions which are in the domain of the action of $\varphi$ on the parts of $\str A$, and the $\theta_{a,v}$ automorphisms to fix the those coordinates of vertex valuation functions which are not in the domain of the action of $\varphi$ on the parts of $\str A$ (and hence can behave like in Example~\ref{ex:semig_bad}).

Fix a partial automorphism $\varphi\colon \str A'\to \str A'$. Looking at the second coordinates (or, equivalently, looking at $\psi^{-1}\varphi\psi$), it induces a partial permutation of $A$ which is part-preserving. Extend it to a part-preserving bijection $\hat\varphi$. Let $\iota$ be the partial permutation of $P$ induced by $\varphi$ and let $\hat\iota$ be the permutation of $P$ induced by $\hat\varphi$. Put $\mathcal P = \dom(\iota)$.

\begin{observation}\label{obs:semig:proj_fine}
For every $\str x = ((A_i, h_0), x, \chi_x) \in \dom(\varphi)$ we have that if $\varphi(\str x) = ((A_j, h_0), y, \chi_y)$ then $\theta_{\hat\varphi}(\str x) = ((A_j, h_0), y, \chi)$ for some $\chi\colon Q\to \mathbb Z_2$.
\end{observation}
\begin{proof}
This follows from the definition of $\theta_{\hat\varphi}$ and the fact that $h_0^{\hat\varphi} = h_0$.
\end{proof}

\begin{definition}
Given $\str x = ((A_p, h_0), x, \chi_x) \in \dom(\varphi)$, denote $\varphi(\str x) = ((A_{\iota(p)}, \allowbreak h_0), u, \chi_u)$ and $\theta_{\hat\varphi}(\str x) = ((A_{\iota(p)}, h_0), u, \chi)$ (cf. Observation~\ref{obs:semig:proj_fine}). We say that $\str x$ \emph{has projection $p$}, and we say that $\str x$ \emph{flips valuation of $a\in P$} if $\chi((A_{\iota(a)}, h_0)) \neq \chi_u((A_{\iota(a)}, h_0))$.
\end{definition}

\begin{prop}\label{prop:equiv}
The following are equivalent for $a, b \in \mathcal P$:
\begin{enumerate}
  \item\label{lem:equiv:1} There is $\str x\in \dom(\varphi)$ which has projection $a$ and flips valuation of $b$,
  \item\label{lem:equiv:3} every $\str x\in \dom(\varphi)$ which has projection $a$ flips valuation of $b$,
  \item\label{lem:equiv:2} there is $\str y\in\dom(\varphi)$ which has projection $b$ and flips valuation of $a$, and
  \item\label{lem:equiv:4} every $\str y\in \dom(\varphi)$ which has projection $b$ flips valuation of $a$.
\end{enumerate}
\end{prop}
\begin{proof}
Note that the equivalences (\ref{lem:equiv:1}) $\iff$ (\ref{lem:equiv:3}) and (\ref{lem:equiv:2}) $\iff$ (\ref{lem:equiv:4}) are actually the same claim, just renaming $a$ and $b$. It is thus enough to prove (\ref{lem:equiv:1}) $\iff$ (\ref{lem:equiv:3}) and (\ref{lem:equiv:1}) $\iff$ (\ref{lem:equiv:2}).

As $a, b \in \mathcal P$, we have $\str x,\str z \in \dom(\varphi)$ with
\begin{align*}
  \str x &= ((A_a, h_0), x, \chi_x),\\
  \str z &= ((A_b, h_0), z, \chi_z),\\
  \varphi(\str x) &= ((A_{\iota(a)}, h_0), u, \chi_u),\\
  \varphi(\str z) &= ((A_{\iota(b)}, h_0), w, \chi_w),\\
  \theta_{\hat\varphi}(\str x) &= ((A_{\iota(a)}, h_0), u, \chi),\\
  \theta_{\hat\varphi}(\str z) &= ((A_{\iota(b)}, h_0), w, \zeta)
\end{align*}
for some $\chi,\zeta\colon Q\to \mathbb Z_2$.

Put $r=0$ if $\iota$ is increasing on $\{a,b\}$ and $r=1$ otherwise.
Since $\varphi$ is a partial automorphism of $\str B$, we know by Observation~\ref{obs:aut_V} that 
$$V(\str x,\str z) + V(\varphi(\str x),\varphi(\str z)) = r.$$
Similarly, since $\theta_{\hat\varphi}$ is an automorphism of $\str B$, we have 
$$V(\str x,\str z) + V(\theta_{\hat\varphi}(\str x),\theta_{\hat\varphi}(\str z)) = r.$$
Summing these two equations (recall that we work in $\mathbb Z_2$), we get
\begin{equation}\label{eqn1}\tag{$\star$}
V(\varphi(\str x),\varphi(\str z)) + V(\theta_{\hat\varphi}(\str x),\theta_{\hat\varphi}(\str z)) = 0.
\end{equation}
If we expand the expressions and use the fact that $h_0$ is constant zero, we obtain
$$\chi_u((A_{\iota(b)}, h_0)) + \chi_w((A_{\iota(a)}, h_0)) + \chi((A_{\iota(b)}, h_0)) + \zeta((A_{\iota(a)}, h_0)) = 0,$$
Which we can rearrange to
$$\chi_u((A_{\iota(b)}, h_0)) + \chi((A_{\iota(b)}, h_0)) = \chi_w((A_{\iota(a)}, h_0)) + \zeta((A_{\iota(a)}, h_0)),$$
or in other words,
$$\chi_u((A_{\iota(b)}, h_0)) = \chi((A_{\iota(b)}, h_0)) \iff \chi_w((A_{\iota(a)}, h_0)) = \zeta((A_{\iota(a)}, h_0)).$$
This just says that $\str x$ flips valuation of $b$ if and only if $\str z$ flips valuation of $a$. Assuming that (\ref{lem:equiv:1}) $\iff$ (\ref{lem:equiv:3}) and (\ref{lem:equiv:2}) $\iff$ (\ref{lem:equiv:4}) are true, this actually proves (\ref{lem:equiv:1}) $\iff$ (\ref{lem:equiv:2}).

It thus remains to prove (\ref{lem:equiv:1}) $\iff$ (\ref{lem:equiv:3}). If $\str x$ is the only vertex from $\dom(\varphi)$ with projection $a$ then it is trivial, so assume that there is $\str y\neq \str x\in \dom(\varphi)$ with
\begin{align*}
  \str y &= ((A_a, h_0), y, \chi_y),\\
  \varphi(\str y) &= ((A_{\iota(a)}, h_0), v, \chi_v),\\
  \theta_{\hat\varphi}(\str y) &= ((A_{\iota(a)}, h_0), v, \xi)
\end{align*}
for some $\xi\colon Q\to \mathbb Z_2$. We will prove that $\str x$ flips the valuation of $b$ if and only if $\str y$ flips the valuation of $b$, from which (\ref{lem:equiv:1}) $\iff$ (\ref{lem:equiv:3}) follows. Note that this is equivalent to the following equality:
$$\chi((A_{\iota(b)}, h_0)) + \chi_u((A_{\iota(b)}, h_0)) + \xi((A_{\iota(b)}, h_0)) + \chi_v((A_{\iota(b)}, h_0)) = 0.$$

By the exact same argument using which we obtained~\eqref{eqn1}, but for $\str y$ and $\str z$, we obtain
$$V(\varphi(\str y),\varphi(\str z)) + V(\theta_{\hat\varphi}(\str y),\theta_{\hat\varphi}(\str z)) = 0,$$
and when we sum this with~\eqref{eqn1} and use the fact that $h_0$ is constant zero, we get
$$\chi_u((A_{\iota(b)}, h_0)) + \chi_v((A_{\iota(b)}, h_0)) + \chi((A_{\iota(b)}, h_0)) + \xi((A_{\iota(b)}, h_0)) = 0,$$
which proves (\ref{lem:equiv:1}) $\iff$ (\ref{lem:equiv:3}).
\end{proof}

Next, we define a set $F$ consisting of pairs $\{a, b\} \subseteq \mathcal P$ such that there is $\str x\in \dom(\varphi)$ which has projection $a$ and flips valuation of $b$. By Proposition~\ref{prop:equiv} this does not depend on the choice of $\str x$ and whether we consider $\{a,b\}$ or $\{b,a\}$ in the definition.

Finally, we define a set $G$ consisting of pairs $(a,v)$ such that $a\in P\setminus \mathcal P$, $v\in A$, $\psi(v) \in \dom(\varphi)$, and $\psi(v)$  flips valuation of $a$.

Let $\theta_F$ be the composition of all $\theta_{\hat\iota(a),\hat\iota(b)}$'s for $\{a, b\}\in F$ with $\hat\iota(a) < \hat\iota(b)$ (by Observation~\ref{obs:semig:com1} $\theta_F$ does not depend on the order of compositions), and let $\theta_G$ be the composition of all $\theta_{\hat\iota(a),\hat\varphi(v)}$'s for $(a,v)\in G$ in an arbitrary order. Put $\theta = \theta_F\theta_G\theta_{\hat\varphi}$.

\begin{lemma}\label{lem:semig:flipsok}
Let $\str v = ((A_a, h_0), v, \chi_v) \in \dom(\varphi)$, and let $b\in P$. The following hold:
\begin{enumerate}
  \item If $b = a$ then $\{a,b\}\notin F$, $(b,v)\notin G$ and $\str v$ does not flip valuation of $b$.
  \item If $b\in \mathcal P$ then $(b,v) \notin G$ and $\{a,b\} \in F$ if and only if $\str v$ flips valuation of $b$.
  \item If $b\notin \mathcal P$ then $\{a,b\}\notin F$ and $(b,v) \in G$ if and only if $\str v$ flips valuation of $b$.
\end{enumerate}
\end{lemma}
\begin{proof}
If $b=a$, then $\str v$ does not flip the valuation of $b$ by the definition of $\theta_{\hat\varphi}$ and the fact that $h_0^{\hat\varphi} = h_0$. Consequently, $\{a,b\}\notin F$ and $(b,v)\notin G$ by Proposition~\ref{prop:equiv}.

If $b\in \mathcal P$ then clearly $(b,v) \notin G$. If $\str v$ flips valuation of $b$ then, by definition of $F$, $\{a,b\} \in F$. On the other hand, if $\{a,b\} \in F$, then, by Proposition~\ref{prop:equiv}, $\str v$ flips valuation of $b$.

If $b\notin \mathcal P$ then clearly $\{a,b\}\notin F$. If $\str v = \psi(v)$ flips valuation of $b$ then, by definition of $G$, $(b,v) \in G$. On the other hand, if $(b,v) \in G$ then, by definition, $\str v$ flips valuation of $b$. 
\end{proof}

\begin{observation}\label{obs:h_h0}
For every $\str x = ((A_a, h_0), x, \chi_x) \in \dom(\varphi)$, if we denote $\theta(\str x) = ((A_b,h'),y,\chi)$ then for every $(A_b,g)\in Q$ it holds that $\chi((A_b,g)) = \chi((A_b,h_0))$.
\end{observation}
\begin{proof}
Fix such $\str x$. Note that $\chi_x$ was constructed such that for every $(A_b,g)\in Q$ we have that $\chi_x((A_b,g)) = \chi_x((A_b,h_0))$. Next, observe that this stays true after applying $\theta_{\hat\varphi}$: This follows from the fact that $h_0$ is $\triangleleft$-minimal, and so the first case in the definition of $\chi'$ for $\theta_{\hat\varphi}$ only happens due to $\hat\iota$ not being monotone (and hence it happens for all part valuation functions for the given part). Next, by Lemma~\ref{lem:semig:flipsok} we know that $G$ only consists of pairs with the first coordinate being from $P\setminus \mathcal P$, so in particular, $a$ is not one of them, and hence this property is preserved by $\theta_G$. Finally, it is easy to see that all $\theta_{a,b}$ functions also preserve this property.
\end{proof}

We are now ready to prove Theorem~\ref{thm:semigeneric}.
\begin{proof}[Proof of Theorem~\ref{thm:semigeneric}]
We only need to prove that $\theta$ is an automorphism of $\str B$ extending $\varphi$. This would imply that $\str B$ is an EPPA-witness for $\str A'$. Clearly, by taking an isomorphism, one gets an EPPA-witness for $\str A$.

The fact that $\theta$ is an automorphism is clear as it is the composition of several automorphisms of $\str B$. In order to prove that $\theta$ extends $\varphi$, pick some $\str x = ((A_a, h_0), x, \chi_x) \in \dom(\varphi)$ and put
$$\varphi(\str x) = ((A_b, h_0), y, \chi_y).$$
By Observation~\ref{obs:semig:proj_fine} we know that
$$\theta_{\hat\varphi}(\str x) = ((A_b, h_0), y, \xi)$$
for some vertex valuation function $\xi\colon Q\to \mathbb Z_2$.
The fact that the automorphisms $\theta_{a,b}$ and $\theta_{a,v}$ do not change the second coordinate and the projection implies that
$$\theta(\str x) = ((A_b, h), y, \chi)$$
for some vertex valuation function $\chi\colon Q\to \mathbb Z_2$. Note also that since $\str x\in \dom(\varphi)$, we know that $a\in \mathcal P$, and hence, by Lemma~\ref{lem:semig:flipsok}, there is no $v\in A$ such that $(a,v)\in G$. Consequently, $h=h_0$, and it remains to verify that $\chi = \chi_y$.

Pick an arbitrary $c\in P$ and a part valuation function $h'\colon A\to \mathbb Z_2$. We want to prove that $\chi((A_{\iota(c)},h')) = \chi_y((A_{\iota(c)},h'))$. From Observation~\ref{obs:h_h0} it follows that $\chi((A_{\iota(c)},h')) = \chi((A_{\iota(c)},h_0))$, and from the construction of $\chi_y$ it follows that $\chi_y((A_{\iota(c)},h'))=\chi_y((A_{\iota(c)},h_0))$. Consequently, it suffices to prove that $\chi((A_{\iota(c)},h_0)) = \chi_y((A_{\iota(c)},h_0))$.

By Lemma~\ref{lem:semig:flipsok} we know that this holds if $c=a$. So $c\neq a$.  Put
$$\theta_G\circ\theta_{\hat\varphi}(\str x) = ((A_b,h_0),y,\xi')$$
for some vertex valuation function $\xi'\colon Q\to \mathbb Z_2$. We now distinguish two cases:

First suppose that $c\in \mathcal P$. In this case there is no $v\in A$ with $(c,v)\in G$ by Lemma~\ref{lem:semig:flipsok}, and so $\xi'((A_{\iota(c)},h_0)) = \xi((A_{\iota(c)},h_0))$. Finally, by Proposition~\ref{prop:equiv} we know that $\{a,c\} \in F$ if and only if $\str x$ flips valuation of $c$, or in other words, if and only if $\xi((A_{\iota(c)},h_0))\neq \chi_y((A_{\iota(c)},h_0))$, and so, indeed, $\chi((A_{\iota(c)},h_0)) = \chi_y((A_{\iota(c)},h_0))$.

So $c\notin\mathcal P$. By Lemma~\ref{lem:semig:flipsok} we know that there is no $d\in P$ with $\{c,d\}\in F$, and hence $\chi((A_{\iota(c)},h_0)) = \xi'((A_{\iota(c)},h_0))$. Hence we need to prove that $\xi'((A_{\iota(c)},h_0)) = \chi_y((A_{\iota(c)},h_0))$. We know that $(c,x)\in G$ if and only if $\str x$ flips valuation of $c$, or in other words, if and only if $\xi((A_{\iota(c)},h_0))\neq \chi_y((A_{\iota(c)},h_0))$, and so, indeed, $\xi'((A_{\iota(c)},h_0)) = \chi_y((A_{\iota(c)},h_0))$. Consequently, $\theta$ indeed extends $\varphi$.
\end{proof}

\section{Ample generics for semigeneric tournaments}\label{sec:sem_ample}
First, we present some basic properties of semigeneric tournaments which first appeared in~\cite{Jasinski2013}.

\begin{observation}\label{obs:semig_equivalences_pair}
Let $\str A$ be a semigeneric tournament and let $X\neq Y\subseteq A$ be a pair of parts of $\str A$. Then there is an equivalence relation $\sim_{X,Y}^\str A \subseteq X^2$ with at most two classes such that for every $x_1,x_2\in X$ and $y\in Y$ we have that the directions of edges $x_1y$ and $x_2y$ are the same if and only if $x_1\sim_{X,Y}^\str A x_2$.
\end{observation}
\begin{proof}
Fix some $y_0\in Y$. Given $x_1,x_2\in X$, define $x_1\sim_{X,Y}^\str A x_2$ if and only if the directions of edges $x_1y_0$ and $x_2y_0$ are the same. The rest follows from Observation~\ref{obs:semig_equiv_weak}.
\end{proof}
In other words, the directions of edges between vertices from two fixed parts are only determined by a pair of equivalence relations of index 2 on those parts.

\begin{definition}
Let $\str A$ be a semigeneric tournament with $k$ parts $A_1,\ldots, A_k$. For every $1\leq i\leq k$, define an equivalence relation $\sim_{A_i}^\str A \subseteq A_i^2$ as follows:
$$\sim_{A_i}^\str A = \bigcap_{1\leq j\leq k, j\neq i} \sim_{A_i,A_j}^\str A,$$
and let $\sim^\str A\subseteq A^2$ be the equivalence relation such that
$$\sim^\str A = \bigcup_{1\leq i\leq k} \sim_{A_i}^\str A.$$
We say that $\str A$ is \emph{saturated} if $\sim^\str A$ has $k2^{k-1}$ equivalence classes. We say that it is \emph{twinless} if every equivalence class of $\sim^\str A$ is a singleton.
\end{definition}

\begin{lemma}\label{lem:sem_autom_sim}
Let $\str A$ be a semigeneric tournament. For every $x,y\in A$ we have that $x\sim^\str A y$ if and only if for every $w\in A$ the following holds
\begin{quote}
  $wx$ is an edge of $\str A$ if and only if $wy$ is, and if they are edges then either both are oriented from $w$, or both are oriented to $w$.
\end{quote}
Consequently, every automorphism of $\str A$ preserves the equivalence relation $\sim^\str A$.
\end{lemma}
\begin{proof}
The statement is clear if $x=y$. If $x\neq y$ are not in the same part then $x\not\sim^\str A y$ and, taking $w = x$, we get that there is no edge between $x$ and $w$ but there is an edge between $y$ and $w$. Thus it remains to consider the case when $x\neq y$ are in the same part of $\str A$.

Let $A_1,\ldots,A_k$ be the parts of $\str A$ and assume without loss of generality that $x,y\in A_1$. If there is $w\in A_j$ such that the edges $xw$ and $yw$ are oriented differently then $x\not\sim_{A_1,A_j}^\str A y$ and thus $x\not\sim^\str A y$. On the other hand, if $x\not \sim^\str A y$ then there is some $j$ such that $x\not \sim_{A_1,A_j}^\str A y$ and this is witnessed by some $w\in A_j$.

Note that the equivalent definition of $\sim^\str A$ from the statement of this lemma is a first-order definition of $\sim^\str A$ from the edge relation of $\str A$, and hence $\sim^\str A$ is preserved by automorphisms of $\str A$.
\end{proof}

\begin{lemma}\label{lem:semig_sat_twin}
Let $\str A$ be a semigeneric tournament with $k$ parts $A_1,\ldots, A_k$.

\begin{enumerate}
  \item $\str A$ is twinless if and only if for every pair of vertices $u \neq v\in A$ from the same part there is $w \in A$ from a different part such that exactly one of the edges $uw$ and $vw$ is oriented to $w$.
  \item The following are equivalent:
    \begin{enumerate}
      \item\label{it:a} $\str A$ is saturated.
      \item\label{it:b} For every $1\leq i,j\leq k$ with $i\neq j$, the equivalence $\sim_{A_i,A_j}^\str A$ has 2 non-empty equivalence classes, and for every $1\leq i\leq k$ and every sequence $(E_j)_{1\leq j\leq k, j\neq i}$ such that $E_j$ is an equivalence class of $\sim_{A_i,A_j}^\str A$, the intersection $\bigcap E_j$ is non-empty.
      \item\label{it:c} For every $1\leq i\leq k$, for every sequence of vertices $(v_j \in A_j : 1\leq j\leq k, j\neq i)$, and every function $f\colon \{1,\ldots,k\}\setminus \{i\} \to \{0,1\}$, there is a vertex $v\in A_i$ such that, for every $1\leq j \leq k, j\neq i$, the edge is oriented from $v$ to $v_j$ if and only if $f(j) = 0$.

  \end{enumerate}
\end{enumerate}
\end{lemma}
\begin{proof}
Clearly, $\sim^\str A$ has a non-singleton equivalence class if and only if there is a pair of vertices connected in the same way to every other vertex.

Similarly, the equivalence (\ref{it:a})~$\iff$~(\ref{it:b}) is easy: $\sim^\str A$ is the union of $k$ equivalences, each of them being an intersection of $k-1$ equivalences of index at most two, hence $\sim^\str A$ can have at most $k2^{k-1}$ equivalence classes and there is only one way how it can happen.

Next, we prove that (\ref{it:b})~$\implies$~(\ref{it:c}). Fix $1\leq i\leq k$ and a sequence of vertices $(v_j \in A_j : 1\leq j\leq k, j\neq i)$. For every $j\neq i$, enumerate the equivalence classes of $\sim_{A_i,A_j}^\str A$ as $E_{j}^0$ and $E_{j}^1$ such that for every $v\in A_i$ we have an edge from $v$ to $v_j$ if and only if $v\in E_{j}^0$. Given a function $f\colon \{1,\ldots,k\}\setminus \{i\} \to \{0,1\}$, there is $v\in \bigcap_{1\leq j\leq k, j\neq i} E_{j}^{f(j)}$ by~(\ref{it:b}) which proves~(\ref{it:c}).

To see that (\ref{it:c})~$\implies$~(\ref{it:a}), fix $1\leq i\leq k$ and an arbitrary sequence of vertices $(v_j \in A_j: 1\leq j\leq k, j\neq i)$. Given $f\colon \{1,\ldots,k\}\setminus \{i\} \to \{0,1\}$, let $v_f\in A_i$ be the vertex given by~(\ref{it:c}). Note that if $f\neq f'$ then $v_f \not\sim_{A_i}^\str A v_{f'}$ (as witnessed by any $v_j$ for $j$ such that $f(j)\neq f'(j)$). This means that $\sim_{A_i}^\str A$ has $2^{k-1}$ equivalence classes for every $1\leq i\leq k$ from which~(\ref{it:a}) follows.
\end{proof}

\begin{prop}\label{prop:semig:factor}
Let $\str A$ be a semigeneric tournament. Let $A/{\sim^\str A}$ be the set of all equivalence classes of $\sim^\str A$ and define a directed graph $\str A/{\sim}$ with vertex set $A/{\sim^\str A}$ such that there is an edge from $[u]$ to $[v]$ in $\str A/{\sim}$ if and only if there is an edge from $u$ to $v$ in $\str A$.
\begin{enumerate}
  \item $\str A/{\sim}$ is a semigeneric tournament, and if $X$ is a part of $\str A$ then $\{[x] : x\in X\}$ is a part of $\str A/{\sim}$.
  \item The map $u\mapsto [u]$ is a homomorphism $\str A\to \str A/{\sim}$ such that if $[u][v]$ is an edge of $\str A/{\sim}$ then $uv$ is an edge of $\str A$.
  \item If $f$ is an automorphism of $\str A$ then its action $[u] \mapsto [f(u)]$ on $\str A/{\sim}$ is well-defined and it is an automorphism of $\str A/{\sim}$.
  \item $\str A/{\sim}$ is twinless.
  \item If $\str A$ is saturated then so is $\str A/{\sim}$.
\end{enumerate}
\end{prop}
\begin{proof}
Lemma~\ref{lem:sem_autom_sim} implies that the existence of edge between $[u]$ and $[v]$ it does not depend on the choice of $u\in [u]$ and $v\in [v]$. Consequently, $\str A/{\sim}$ is a semigeneric tournament, it has the same parts as $\str A$, automorphisms of $\str A$ give automorphisms of $\str A/{\sim}$, and the map $u\mapsto [u]$ has the desired properties. Finally, $\str A/{\sim}$ is, by definition, twinless.

Suppose now that $\str A$ is saturated. Enumerate the parts of $\str A$ as $A_1,\ldots, A_k$, and, for convenience, denote the parts of $\str A/{\sim}$ as $[A_1],\ldots,[A_k]$ such that $[A_i]$ consists of the equivalence classes of elements of $A_i$.

Pick some $1\leq i\leq k$, a sequence of vertices $([v_j] \in [A_j] : 1\leq j\leq k, j\neq i)$, and a function $f\colon \{1,\ldots,k\}\setminus \{i\} \to \{0,1\}$. This gives us a sequence of vertices $(v_j \in A_j : 1\leq j\leq k, j\neq i)$, and by saturation of $\str A$ we get a vertex $v\in A_i$ such that, for every $1\leq j \leq k, j\neq i$, the edge is oriented from $v$ to $v_j$ if and only if $f(j) = 0$. Consequently, we have $[v]\in [A_i]$ witnessing that $\str A/{\sim}$ is saturated.
\end{proof}

\begin{prop}\label{prop:ample_twinless}
For every finite semigeneric tournament $\str A$ there exists a finite saturated twinless semigeneric tournament $\str B$ which is an EPPA-witness for $\str A$.
\end{prop}
\begin{proof}
Let $\str B$ be the semigeneric tournament constructed in Section~\ref{sec:semig} for $\str A$. First, we prove that $\str B$ is saturated:

By Lemma~\ref{lem:semig_sat_twin} we need to prove that for every $(A_i,h)\in Q$, for every sequence $(\str v_{(A_j,g)})_{(A_j,g) \in Q\setminus \{(A_i,h)\}}$, and for every function $f\colon Q\setminus \{(A_i,h)\}\to \{0,1\}$, there is $\str v\in B$ connected to $(\str v_{(A_j,g)})_{(A_j,g) \in Q\setminus \{(A_i,h)\}}$ according to $f$. Pick an arbitrary $v\in A_i$ and observe that there is a vertex valuation function $\chi\colon Q\to\mathbb Z_2$ such that $\str v = ((A_i,h),v,\chi)$ has precisely the desired edges to $(\str v_{(A_j,g)})_{(A_j,g) \in Q\setminus \{(A_i,h)\}}$, as for each $(A_j,g) \in Q\setminus \{(A_i,h)\}$ we can make an independent choice for $\chi((A_j,g))$ and determine the direction of the edge between $\str v$ and $\str v_{(A_j,g)}$. Consequently, $\str B$ is indeed saturated.

\medskip

Recall the definition of $\psi\colon \str A \to \str B$ and that we put $\str A' = \psi(\str A)$ (so that $\str B$ is an EPPA-witness for $\str A'$). We claim that for every $u\neq v\in A$ it holds that $\psi(u) \not\sim^\str B \psi(v)$. Indeed: If $u$ and $v$ are from different parts of $\str A$ then $\psi(u)$ and $\psi(v)$ are from different parts of $\str B$, hence they are not equivalent.

If $u,v\in A_i$ for some $i$, denote $\psi(u) = ((A_i, h_0), u, \chi_u)$ and $\psi(v) = ((A_i, h_0), v, \chi_v)$. From the definition of $\chi_u$ and $\chi_v$ (see Section~\ref{sec:semig:const}) we have that, for every part valuation function $h\colon A\to\mathbb Z_2$, $\chi_u((A_i,h)) = \chi_v((A_i,h)) = 0$. Let $h\colon A\to\mathbb Z_2$ be the part valuation function such that $h(u) = 1$ and $h(x) = 0$ for every $x\neq u$, and let $\str x$ be an arbitrary vertex of $\str B$ from the part $(A_i,h)$. It follows that the directions of edges $\str x\psi(u)$ and $\str x\psi(v)$ disagree, hence $\psi(u) \not\sim^\str B \psi(v)$.

\medskip

Define function $\psi'\colon  \str A\to \str B/{\sim}$ by $\psi'(x) = [\psi(x)]$. Note that this is a composition of an embedding $\psi$ and the function $u\mapsto [u]$ which is injective on $A'$ and which, by Proposition~\ref{prop:semig:factor}, is a non-edge-preserving homomorphism, hence $\psi'$ is an embedding. We will show that $\str B/{\sim}$ is a finite saturated twinless semigeneric tournament which is an EPPA-witness for $\psi'(\str A)$.

The fact that it is a finite saturated twinless semigeneric tournament follows directly from Proposition~\ref{prop:semig:factor}. Let $\varphi$ be a partial automorphism of $\psi'(\str A)$. By taking an isomorphism, we can consider it to be a partial automorphism of $\psi(\str A)$ and extend it to an automorphism $\theta\colon \str B\to \str B$. Let $\theta'\colon \str B/{\sim} \to \str B/{\sim}$ satisfy $\theta'([x]) = [\theta(x)]$. By Proposition~\ref{prop:semig:factor}, this is a well-defined automorphism of $\str B/{\sim}$, and it is easy to see that it extends $\varphi$.
\end{proof}

\begin{lemma}\label{lem:semigeneric_apa}
Semigeneric tournaments have APA over saturated twinless structures.
\end{lemma}
\begin{proof}
We will proceed analogously to Lemma~\ref{lem:omegapart_apa}. Fix $\str A$, $\str B_1$ and $\str B_2$ such that $\str A$ is a substructure of $\str B_1$ and $\str B_2$, $\str A$ is twinless and saturated, and $B_1\cap B_2 = A$.

Let $A_1,\ldots, A_k$ be the parts of $\str A$. Since $\str A$ is saturated and twinless, for every equivalence class $E$ of $\sim^\str A$ we can denote $E = \{ a_{E} \}$.

Put $C = B_1\cup B_2$ and define a partition of $C$ using the partitions of $B_1$ and $B_2$, merging only those parts with non-empty intersections. Let $C^1,\ldots, C^n$ be the resulting partition.

Next, we will define a semigeneric tournament $\str C$ with vertex set $C$ and parts $C^1,\ldots, C^n$. Given $i\neq j$, $x\in C^i$ and $y\in C^j$, we define the orientation of the edge $xy$ as follows:
\begin{enumerate}
\item\label{sem_apa:1} If $x,y\in B_1$ or $x,y\in B_2$ then we copy the orientation from $\str B_1$ resp. $\str B_2$.
\item\label{sem_apa:2} If $C^i \cap A = C^j\cap A = \emptyset$ then we orient the edge from $x$ to $y$ if and only if $x\in B_1$ and $y\in B_2$. (Otherwise $x\in B_2$ and $y\in B_1$ and we orient the edge from $y$ to $x$; in other words, we orient all these edges from $\str B_1$ to $\str B_2$.)
\item\label{sem_apa:3} If there are $u\in C^i\cap A$ and $v\in C^j\cap A$ (and hence the edges $uv$, $uy$ and $vx$ have been already oriented by~(\ref{sem_apa:1})), we orient the edge so that the four vertices $x,y,u,v$ induce a semigeneric tournament (there is exactly one such choice).
\item\label{sem_apa:4} In the remaining case we have without loss of generality $C^j\subseteq B_2\setminus A$, and some $v\in C^i\cap A$. Moreover, $x\in B_1\setminus A$ as otherwise the edge would have been oriented by~(\ref{sem_apa:1}). Let $\str B_1'$ be the substructure of $\str B_1$ induced on the parts with non-empty intersection with $\str A$. Since $\str A$ is saturated, so is $\str B_1'$, and moreover, equivalence class of $\sim^{\str B_1'}$ contains a unique vertex of $\str A$, because $\str A$ is twinless and saturated.

Consequently, there is a unique vertex $a_{A\cap E}\in A$ such that $a_{A\cap E} \sim^{\str B_1'} x$. Now, orient the edge $xy$ from $x$ to $y$ if and only if the edge $a_{A\cap E}y$ is oriented from $a_{A\cap E}$ to $y$.
\end{enumerate}

Clearly, the cases~(\ref{sem_apa:1})--(\ref{sem_apa:4}) are disjoint and together cover all possibilities, hence $\str C$ is an $n$-partite tournament with parts $C^1,\ldots,C^n$. Moreover, $\str B_1$ and $\str B_2$ are substructures of $\str C$, hence $\str C$ is an amalgam of $\str B_1$ and $\str B_2$ over $\str A$. It remains to prove that $\str C$ is a semigeneric tournament and that it has all the desired automorphisms.

\medskip

First we will see that $\str C$ is a semigeneric tournament. To see this, pick arbitrary vertices $x_1,x_2,y_1,y_2\in C$ such that $x_1$ and $x_2$ are from the same part and $y_1$ and $y_2$ are also from the same part. It is easy to see that if the edges in this quadruple have only been determined using cases~(\ref{sem_apa:1})--(\ref{sem_apa:3}) then this quadruple indeed satisfies the semigeneric condition.

In the remaining case we have $x_1,x_2\in C^i$, $y_1,y_2\in C^j$, and, without loss of generality, $C^j\subseteq B_2 \setminus A$, $C^i\subseteq B_1$, and $C^i\cap A\neq \emptyset$. If $x_1,x_2\in A$ then we are in case~(\ref{sem_apa:1}), hence at least one of $x_1$ and $x_2$ is from $B_1\setminus A$. If $x_1\in B_1\setminus A$, we know that its edges have been copied from a suitable $a_E\in A$, thus we can replace $x_1$ by this $a_E$ and verify the semigeneric condition for $a_E,x_2,y_1,y_2$ instead. We can do the same for $x_2$ and then we get a quadruple of vertices from $\str B_2$ for which we know that the semigeneric condition is satisfied.

\medskip

Let $f_1\colon \str B_1\to \str B_1$ and $f_2\colon \str B_2\to \str B_2$ be automorphisms such that $f_1\restriction_A = f_2\restriction_A$. Put $f = f_1\cup f_2$. Similarly as in Lemma~\ref{lem:omegapart_apa} we can see that $f$ is a part-preserving bijection $C\to C$. It remains to see that if $x\in B_1\setminus A$ and $y\in B_2\setminus A$ are from different parts then there is an edge from $x$ to $y$ in $\str C$ if and only if there is an edge from $f(x)$ to $f(y)$ in $\str C$. This is clear if the direction of the edge $xy$ (and hence also of $f(x)f(y)$) has been determined by cases~(\ref{sem_apa:2}) or~(\ref{sem_apa:3}) -- one just needs to observe that $f$ preserves whether a part of $\str C$ has a non-empty intersection with $A$.

So the direction of the edge $xy$ (and hence also of $f(x)f(y)$) has been determined by case~(\ref{sem_apa:4}). Suppose that $x\in C^i$, $y\in C^j$ and assume without loss of generality that $C^j\subseteq B_2\setminus A$, $C^i\cap A\neq \emptyset$, and $x\in B_1\setminus A$. Let $\str B_1'$ be the substructure of $\str B_1$ induced on the parts with a non-empty intersection with $\str A$.

Let $E$ be the equivalence class of $\sim^{\str B_1'}$ containing $x$. By Lemma~\ref{lem:sem_autom_sim}, $f(E)$ is the equivalence class of $\sim^{\str B_1'}$ containing $f(x)$, and we know that $E\cap A = \{a_E\}$ and $f(E)\cap A = \{a_{f(E)}\} = \{f(a_E)\}$. Since $f$ preserves $A$ we get that $f(C^j) \cap A = \emptyset$, and thus the direction of the edge $f(x)f(y)$ was also determined by case~(\ref{sem_apa:4}). Consequently, the direction of the edge $xy$ was copied from the direction of $a_Ey$ and the direction of the edge $f(x)f(y)$ was copied from the direction of $f(a_E)f(y)$. But $a_E,y\in \str B_2$ and thus $f(a_E) = f_2(a_E)$ and $f(y) = f_2(y)$, and since $f_2$ is an automorphism, we get that the orientations of $a_Ey$ and $f(a_E)f(y)$ agree. Consequently, $f$ is indeed an automorphism $\str C\to \str C$.
\end{proof}

We can now prove Theorem~\ref{thm:semig_ample}.
\begin{proof}[Proof of Theorem~\ref{thm:semig_ample}]
  Lemma~\ref{lem:semigeneric_apa} gives APA for semigeneric tournaments over saturated twinless structures, and the empty structure is both saturated and twinless. Proposition~\ref{prop:ample_twinless} tells us that we can always construct saturated twinless EPPA-witnesses. Theorem~\ref{thm:kr} and Proposition~\ref{prop:ample_apa} thus imply ample generics for the \Fraisse{} limit of the class of all finite semigeneric tournaments.
\end{proof}

\section{Conclusion}\label{sec:conclusion}
Question~\ref{q:tournaments} is a relaxation of the question whether directed graphs with no independent set of size $k$  (which for $k=2$ means tournaments) have EPPA. Note that the following further relaxation of Question~\ref{q:tournaments} is trivially true:
\begin{observation}\label{obs:trivial}
For every $k\geq 2$ and every $n$ there is $\ell = \ell(k, n)$ such that for every directed graph $\str A$ on $n$ vertices which contains no independent set of size $k$ there is a directed graph $\str B$ which contains no independent set of size $\ell$ such that $\str A\subseteq \str B$ and every partial automorphism of $\str A$ extends to an automorphism of $\str B$.
\end{observation}
\begin{proof}[Proof of Observation~\ref{obs:trivial}]
EPPA for directed graphs has been proved by Herwig in 1995~\cite{Herwig1995}. Given $k$ and $n$, there are only finitely many directed graphs on $n$ vertices with no independent set of size $k$. Put $\ell(k,n)$ to be the largest independent set in an EPPA-witness for one of these graphs produced by Herwig.
\end{proof}
We remark that in Section~4 of~\cite{Hubicka2018EPPA}, Hubi\v cka, Kone\v cn\'y, and Ne\v set\v ril prove the following theorem (or rather a more general version of it):
\begin{theorem}[\cite{Hubicka2018EPPA}]
Let $L$ be a finite relational language. Then for every $n$ there is an $L$-structure $\str B$ such that $\str B$ is an EPPA-witness for every $L$-structure $\str A$ on at most $n$ vertices.
\end{theorem}
In~\cite{BradleyEPPAnumbers} it is shown that one can in fact get $\str B$ with $n3^{n-1}$ vertices and largest independent set of size $3^{n-1}$, hence proving $\ell(k,n)\leq 3^{n-1}$.

While Observation~\ref{obs:trivial} is trivial, Question~\ref{q:tournaments} seems as hard to approach as EPPA for directed graphs without large independent sets using the current methods. On the other hand, it seems to non-trivially relax the group-theoretic constraints and thus it might be a fruitful direction of attack: For example, tournaments ($k=2$) have no automorphisms of even order, but this is no longer true for directed graphs with no independent set of size 3. (However, it would be interesting to understand, group-theoretically, how much this is no longer true; of course, the edgeless graph on 2 vertices is such an example, or more generally, double covers of tournaments, but do we get (significantly) more than this in relation to potential EPPA-witnesses?)

\medskip

In our proof of ample generics for semigeneric tournaments we only proved APA over saturated twinless structures. To our best knowledge it is open whether this was necessary (that is, whether semigeneric tournaments have APA):
\begin{question}
Does the class of all finite semigeneric tournaments have the amalgamation property with automorphisms?
\end{question}
We conjecture that the answer is negative.

\medskip

In Proposition~\ref{prop:ample_twinless} we needed to slightly massage the construction from Section~\ref{sec:semig} to obtain an EPPA-witness with better properties. In particular, it shows that the original EPPA-witness is not optimal. There are some obvious inefficiencies (part valuation functions do not really need to valuate vertices from their own part, and similarly, vertex valuation functions do not have to valuate their own part), but the arbitrary choice of $v\in A_i$ in the argument that $\str B$ is saturated shows that one can save more. It is not immediately obvious to us how to construct such a smaller EPPA-witness without taking a quotient. It would be interesting to see such a construction as, perhaps, it may have some better properties, maybe it could be useful for coherent EPPA which we discuss next:

\subsection{Coherent EPPA}
There is a strengthening of EPPA called \emph{coherent EPPA} which was introduced by Siniora and Solecki~\cite{Siniora} with the goal of generalizing the proof of Bhattacharjee and Macpherson that the automorphism group of the countable random graph contains a dense locally finite subgroup~\cite{Bhattacharjee2005}: If $\str B$ is an EPPA-witness for $\str A$, we say that it is \emph{coherent} if there is a map $\Psi$ from partial automorphisms of $\str A$ to automorphisms of $\str B$ such that $\Psi(f)$ extends $f$, and $\Psi(gf) = \Psi(g)\Psi(f)$ for every pair $f,g$ of partial automorphisms of $\str A$ with $\range(f) = \dom(g)$. A class $\mathcal C$ has coherent EPPA if for every $\str A\in \mathcal C$ there is a coherent EPPA-witness $\str B\in\mathcal C$ for $\str A$.

Semigeneric and $n$-partite tournaments are, in addition to two-graphs~\cite{eppatwographs}, the only classes for which EPPA is known but coherent EPPA is open (and there are no known classes which provably have EPPA but not coherent EPPA). We thus ask:

\begin{question}
Do semigeneric and/or $n$-partite tournaments have coherent EPPA ($n\in \{2,\ldots,\omega\}$)?
\end{question}

In the case of two-graphs, even though coherent EPPA is open, it was possible to prove the group theoretic consequence of coherent EPPA, namely that the automorphism group of the generic countable two-graph contains a dense locally finite subgroup. However, for semigeneric and $n$-partite tournaments this is also open.

\begin{question}
Do the automorphism groups of the generic countable semigeneric and/or $n$-partite tournament contain a dense locally finite subgroup ($n\in \{2,\ldots,\omega\}$)?
\end{question}

Note that our constructions do not seem to lead to coherent EPPA-witnesses: In the definition of $\theta_\pi$, namely in the definition of $\chi'$, the choice to use $1+\chi(y)$ if $x < y$ and $\pi(x) > \pi(y)$ (in the $n$-partite case, for the semigeneric case it behaves in the same way) is not canonical, one could also use it if and only if $x > y$ and $\pi(x) < \pi(y)$. Because of this, it is not true that $\theta_\pi\theta_\sigma = \theta_{\pi\sigma}$: Consider, for example, the transposition $\pi = (xy)$ which fixes all other elements except for some $x < y\in A$. Clearly, $\pi = \pi^{-1}$. However, $\theta_\pi$ flips the $y$-th entry of all valuation functions with projection $x$, and so $\theta_\pi^2 = \theta_{x,y}$. We were not able to find a workaround for this problem.

Additionally, for the semigeneric tournament a similar issue arises also for $\theta_{a,v}$.

\subsection{Profinite topology}
Herwig and Lascar proved that EPPA for tournaments is equivalent to a statement about free groups:
\begin{theorem}[Herwig--Lascar~\cite{herwig2000}]
The following statements are equivalent:
\begin{enumerate}
  \item The class of all finite tournaments has EPPA.
  \item For every $n\geq 1$ and every finitely generated subgroup $H$ of $F_n$, the free group on $n$ elements, $H$ is closed in the odd-adic topology on $F_n$ if and only if, for every $a\in F_n$, if $a^2\in H$ then $a\in H$.
\end{enumerate}
\end{theorem}
Here, the \emph{odd-adic} topology on $F_n$ is given by the following basis of open sets:
$$\{gH : g\in F_n, H\text{ is a normal subgroup of $F_n$ of odd index}\}.$$
This is closely related to the more well-known \emph{profinite topology} of $F_n$ whose basis of open sets is:
$$\{gH : g\in F_n, H\text{ is a subgroup of $F_n$ of finite index}\},$$
see e.g.~\cite{Hall1950,Ribes1993}.

Another connection of this flavour has been used by Huang, Pawliuk, Sabok, and Wise to disprove EPPA for the so-called $L$-hypertournaments, where $L$ is a set of prime numbers~\cite{Sabok}. It would be interesting to see if our results imply something interesting about profinite-like topologies.

\section*{Acknowledgments}
The authors would like to thank David Evans for helpful comments, and to both anonymous referees for their detailed reading and comments which improved the quality of this paper.

\bibliography{ramsey.bib}

\newcommand{\etalchar}[1]{$^{#1}$}
\begin{thebibliography}{JLNVTW14}

\bibitem[ABWH{\etalchar{+}}26]{Aranda2017}
Andres Aranda, David Bradley-Williams, Jan Hubi{\v c}ka, Miltiadis Karamanlis,
  Michael Kompatscher, Mat{\v e}j Kone{\v c}n{\'y}, and Micheal Pawliuk.
\newblock Ramsey expansions of metrically homogeneous graphs.
\newblock {\em European Journal of Combinatorics}, 132:in press,
  arXiv:1707.02612, 2026.

\bibitem[BdlHV08]{Bekka2008}
Bachir Bekka, Pierre de~la Harpe, and Alain Valette.
\newblock {\em Kazhdan's {P}roperty ({T})}.
\newblock New Mathematical Monographs. Cambridge University Press, 2008.

\bibitem[BM05]{Bhattacharjee2005}
Meenaxi Bhattacharjee and Dugald Macpherson.
\newblock A locally finite dense group acting on the random graph.
\newblock {\em Forum Mathematicum}, 17(3):513--517, 2005.

\bibitem[BWCHK25]{BradleyEPPAnumbers}
David Bradley-Williams, Peter~J. Cameron, Jan Hubička, and Matěj Konečný.
\newblock Eppa numbers of graphs.
\newblock {\em Journal of Combinatorial Theory, Series B}, 170:203--224, 2025.

\bibitem[Che98]{Cherlin1998}
Gregory Cherlin.
\newblock {\em The Classification of Countable Homogeneous Directed Graphs and
  Countable Homogeneous $N$-tournaments}.
\newblock Number 621 in Memoirs of the American Mathematical Society. American
  Mathematical Society, 1998.

\bibitem[Che22]{Cherlin2013}
Gregory Cherlin.
\newblock {\em Homogeneous Ordered Graphs, Metrically Homogeneous Graphs, and
  Beyond}, volume~1 of {\em Lecture Notes in Logic}.
\newblock Cambridge University Press, 2022.

\bibitem[Con19]{Conant2015}
Gabriel Conant.
\newblock Extending partial isometries of generalized metric spaces.
\newblock {\em Fundamenta Mathematicae}, 244:1--16, 2019.

\bibitem[CSS99]{Cherlin1999}
Gregory Cherlin, Saharon Shelah, and Niandong Shi.
\newblock Universal graphs with forbidden subgraphs and algebraic closure.
\newblock {\em Advances in Applied Mathematics}, 22(4):454--491, 1999.

\bibitem[EHKN20]{eppatwographs}
David~M. Evans, Jan Hubi{\v{c}}ka, Mat{\v {e}}j Kone{\v {c}}n{\'{y}}, and
  Jaroslav Ne\v{s}et\v{r}il.
\newblock E{P}{P}{A} for two-graphs and antipodal metric spaces.
\newblock {\em Proceedings of the American Mathematical Society},
  148:1901--1915, 2020.

\bibitem[EHN21]{Evans3}
David~M. Evans, Jan Hubi{\v{c}}ka, and Jaroslav Ne{\v{s}}et{\v{r}}il.
\newblock Ramsey properties and extending partial automorphisms for classes of
  finite structures.
\newblock {\em Fundamenta Mathematicae}, 253:121--153, 2021.

\bibitem[Fra53]{Fraisse1953}
Roland Fra{\"\i}ss{\'e}.
\newblock Sur certaines relations qui g\'en\'eralisent l'ordre des nombres
  rationnels.
\newblock {\em Comptes Rendus de l'Academie des Sciences}, 237:540--542, 1953.

\bibitem[Fra86]{Fraisse1986}
Roland Fra{\"\i}ss{\'e}.
\newblock {\em Theory of relations}.
\newblock Studies in logic and the foundations of mathematics. North-Holland,
  1986.

\bibitem[Hal50]{Hall1950}
Marshall Hall.
\newblock A topology for free groups and related groups.
\newblock {\em Annals of Mathematics}, 52:127, 1950.

\bibitem[Her95]{Herwig1995}
Bernhard Herwig.
\newblock Extending partial isomorphisms on finite structures.
\newblock {\em Combinatorica}, 15(3):365--371, 1995.

\bibitem[Her98]{herwig1998}
Bernhard Herwig.
\newblock Extending partial isomorphisms for the small index property of many
  $\omega$-categorical structures.
\newblock {\em Israel Journal of Mathematics}, 107(1):93--123, 1998.

\bibitem[HHLS93]{hodges1993b}
Wilfrid Hodges, Ian Hodkinson, Daniel Lascar, and Saharon Shelah.
\newblock The small index property for $\omega$-stable $\omega$-categorical
  structures and for the random graph.
\newblock {\em Journal of the London Mathematical Society}, 2(2):204--218,
  1993.

\bibitem[HJKS19]{HubickaSemigenericAMUC}
Jan Hubi\v{c}ka, Colin Jahel, Mat{\v e}j Kone{\v c}n{\'y}, and Marcin Sabok.
\newblock Extending partial automorphisms of $n$-partite tournaments.
\newblock {\em Acta Mathematica Universitatis Comenianae}, 88(3):807--811,
  2019.
\newblock Extended abstract for Eurocomb 2019.

\bibitem[HK26]{hubicka2025twenty}
Jan Hubi{\v{c}}ka and Mat{\v{e}}j Kone{\v{c}}n{\'y}.
\newblock Twenty years of {N}e{\v s}et{\v r}il's classification programme of
  {R}amsey classes.
\newblock {\em Computer Science Review}, 59:100814, 2026.

\bibitem[HKN18]{Hubicka2017sauer}
Jan Hubi{\v{c}}ka, Mat{\v {e}}j Kone{\v {c}}n{\'{y}}, and Jaroslav
  Ne\v{s}et\v{r}il.
\newblock Semigroup-valued metric spaces: {R}amsey expansions and {E}{P}{P}{A}.
\newblock In preparation, 2018.

\bibitem[HKN19]{Hubicka2018metricEPPA}
Jan Hubi{\v{c}}ka, Mat{\v{e}}j Kone{\v{c}}n{\'y}, and Jaroslav
  Ne{\v{s}}et{\v{r}}il.
\newblock A combinatorial proof of the extension property for partial
  isometries.
\newblock {\em Commentationes Mathematicae Universitatis Carolinae},
  60(1):39--47, 2019.

\bibitem[HKN22]{Hubicka2018EPPA}
Jan Hubi{\v{c}}ka, Mat{\v{e}}j Kone{\v{c}}n{\'{y}}, and Jaroslav
  Ne{\v{s}}et{\v{r}}il.
\newblock All those {E}{P}{P}{A} classes (strengthenings of the
  {H}erwig--{L}ascar theorem).
\newblock {\em Transactions of the American Mathematical Society},
  375(11):7601--7667, 2022.

\bibitem[HL00]{herwig2000}
Bernhard Herwig and Daniel Lascar.
\newblock Extending partial automorphisms and the profinite topology on free
  groups.
\newblock {\em Transactions of the American Mathematical Society},
  352(5):1985--2021, 2000.

\bibitem[HO03]{hodkinson2003}
Ian Hodkinson and Martin Otto.
\newblock Finite conformal hypergraph covers and {G}aifman cliques in finite
  structures.
\newblock {\em Bulletin of Symbolic Logic}, 9(03):387--405, 2003.

\bibitem[HPSW19]{Sabok}
Jingyin Huang, Micheal Pawliuk, Marcin Sabok, and Daniel~T Wise.
\newblock The {H}rushovski property for hypertournaments and profinite
  topologies.
\newblock {\em Journal of the London Mathematical Society}, 100(3):757--774,
  2019.

\bibitem[Hru92]{hrushovski1992}
Ehud Hrushovski.
\newblock Extending partial isomorphisms of graphs.
\newblock {\em Combinatorica}, 12(4):411--416, 1992.

\bibitem[JLNVTW14]{Jasinski2013}
Jakub Jasi{\'n}ski, Claude Laflamme, Lionel Nguyen Van~Th{\'e}, and Robert
  Woodrow.
\newblock Ramsey precompact expansions of homogeneous directed graphs.
\newblock {\em The Electronic Journal of Combinatorics}, 21(4):\#4.42, 2014.

\bibitem[Kon19]{Konecny2018b}
Mat{\v e}j Kone{\v c}n{\'y}.
\newblock Semigroup-valued metric spaces.
\newblock Master's thesis, Charles University, 2019.
\newblock arXiv:1810.08963.

\bibitem[Kon20]{Konecny2019a}
Mat{\v e}j Kone{\v c}n{\'y}.
\newblock Extending partial isometries of antipodal graphs.
\newblock {\em Discrete Mathematics}, 343(1):111633, 2020.

\bibitem[KR07]{Kechris2007}
Alexander~S. Kechris and Christian Rosendal.
\newblock Turbulence, amalgamation, and generic automorphisms of homogeneous
  structures.
\newblock {\em Proceedings of the London Mathematical Society}, 94(2):302--350,
  2007.

\bibitem[Lac84a]{lachlan1984countable}
Alistair~H. Lachlan.
\newblock Countable homogeneous tournaments.
\newblock {\em Transactions of the American Mathematical Society},
  284(2):431--461, 1984.

\bibitem[Lac84b]{Lachlan1984}
Alistair~H. Lachlan.
\newblock On countable stable structures which are homogeneous for a finite
  relational language.
\newblock {\em Israel Journal of Mathematics}, 49(1-3):69--153, 1984.

\bibitem[LW80]{Lachlan1980}
Alistair~H. Lachlan and Robert~E. Woodrow.
\newblock Countable ultrahomogeneous undirected graphs.
\newblock {\em Transactions of the American Mathematical Society},
  262(1):51--94, 1980.

\bibitem[Ott20]{otto2017}
Martin Otto.
\newblock Amalgamation and symmetry: From local to global consistency in the
  finite.
\newblock arXiv:1709.00031, 2020.

\bibitem[PS20]{PawliukSokic16}
Micheal Pawliuk and Miodrag Soki{\'c}.
\newblock Amenability and unique ergodicity of automorphism groups of countable
  homogeneous directed graphs.
\newblock {\em Ergodic Theory and Dynamical Systems}, 40(5):1351--1401, 2020.

\bibitem[RZ93]{Ribes1993}
Luis Ribes and Pavel~A Zalesskii.
\newblock On the profinite topology on a free group.
\newblock {\em Bulletin of the London Mathematical Society}, 25(1):37--43,
  1993.

\bibitem[Sin17]{Siniora2}
Daoud Siniora.
\newblock {\em Automorphism Groups of Homogeneous Structures}.
\newblock PhD thesis, University of Leeds, March 2017.

\bibitem[Sol05]{solecki2005}
S{\l}awomir Solecki.
\newblock Extending partial isometries.
\newblock {\em Israel Journal of Mathematics}, 150(1):315--331, 2005.

\bibitem[SS20]{Siniora}
Daoud Siniora and S{\l}awomir Solecki.
\newblock Coherent extension of partial automorphisms, free amalgamation, and
  automorphism groups.
\newblock {\em The Journal of Symbolic Logic}, 85(1):199--223, 2020.

\bibitem[Tru92]{Truss1992}
J.~K. Truss.
\newblock Generic automorphisms of homogeneous structures.
\newblock {\em Proceedings of the London Mathematical Society},
  s3-65(1):121--141, 1992.

\bibitem[Ver08]{vershik2008}
Anatoly~M. Vershik.
\newblock Globalization of the partial isometries of metric spaces and local
  approximation of the group of isometries of {U}rysohn space.
\newblock {\em Topology and its Applications}, 155(14):1618--1626, 2008.

\end{thebibliography}
\end{document}